\documentclass{amsart}
\usepackage{lipsum,color}
\usepackage{amsfonts}
\usepackage{graphicx}
\usepackage{epstopdf}
\usepackage{algorithmic}
\ifpdf
\DeclareGraphicsExtensions{.eps,.pdf,.png,.jpg}
\else
\DeclareGraphicsExtensions{.eps}
\fi

\usepackage{amsmath,amssymb}
\usepackage{mathrsfs}
\usepackage{dsfont} %\mathds{} fonts. Alternative to \mathbb{}, and additionally gives indicator function.
\usepackage{mathtools} %Add-on to amsmath package
\usepackage{esint} %More control of integrals
\usepackage{xfrac} %Gives inline fraction command \sfrac{}{}
\usepackage{defs} %Custom style file
\usepackage{fonts} %Custom style file

\usepackage{tikz}
\usetikzlibrary{arrows,calc}

\usepackage[shortlabels]{enumitem} %For doing almost anything with lists

\usepackage[noadjust]{cite} %Sorts and formats citations
\usepackage{cleveref}
\newtheorem{theorem}{Theorem}
\newtheorem{proposition}{Proposition}
\newtheorem{corollary}{Corollary}
\newtheorem{lemma}{Lemma}
\newtheorem{defn}{Definition}
\newtheorem{assumption}{Assumption}
\newtheorem{remark}{Remark}
\numberwithin{equation}{section}

\usepackage{amsopn}

\begin{document}

\title[AC schemes via Gamma-convergence]{
	%Asymptotically compatible discretizations/ of nonlocal problems with local boundary conditions
    Asymptotically compatible schemes for 
    %parametrized 
    nonlinear variational models via Gamma-convergence 
    and applications to %nonlinear 
    nonlocal problems}
% Authors: full names plus addresses.
\author{Qiang Du}
\address{Department of Applied Physics and Applied Mathematics, and Data Science Institute, Columbia University, New York, NY 10027, USA.}
\curraddr{}
\email{qd2125@columbia.edu}

%    author two information
\author{James M. Scott}
\address{Department of Applied Physics and Applied Mathematics, Columbia University, New York, NY 10027, USA.}
\curraddr{}
\email{jms2555@columbia.edu}

\author{Xiaochuan Tian}
\address{Department of Mathematics, University of California, San Diego, CA 92093, USA.}\curraddr{}
\email{xctian@ucsd.edu}

\thanks{This work is supported in part 
			by the NSF DMS-2309245, DMS-1937254,  DMS-2111608, DMS-2240180.}

%    \subjclass is required.

%\date{\today}

	\begin{abstract}
		We present a study on asymptotically compatible Galerkin discretizations for a class of parametrized nonlinear variational problems.
        %The generic
        The abstract
        analytical framework is based on
        %uses
        variational convergence, or Gamma-convergence.
        %, rather than the strong convergence of operators.
        We demonstrate the broad applicability of the theoretical framework by developing
        % apply this program to obtain 
        asymptotically compatible finite element discretizations of %a class of 
        some representative nonlinear nonlocal variational problems on a bounded domain. These include nonlocal nonlinear problems
        with classically-defined, local boundary constraints through heterogeneous localization at the boundary, 
        %Additionally,
        as well as 
        %we consider an application to parameterized 
        nonlocal problems posed on parameter-dependent domains.
        %the linear regime 
        %to connect with
        %in this application in the context of 
        %previous studies of asymptotic compatibility
        %in the linear regime.
	\end{abstract}
	
	% REQUIRED

\maketitle

	% REQUIRED

	% REQUIRED

	\section{Introduction}\label{sec:intro}
Our study focuses on the Galerkin finite-dimensional approximations of a class of %linear and nonlinear 
variational problems for an energy functional associated with a scalar modeling parameter. The parameter can vary continuously, and in particular, can reach an asymptotic limit. In the asymptotic regime, appropriate conditions on the energy functional for various parameter values are assumed, including the Gamma-convergence to the limit functional. We then examine the analogous convergence properties 
%concerning 
for the Galerkin approximations.

Although the approximations of parametrized problems have been a popular subject of computational mathematics with broad applications, the present
%This 
study is particularly motivated  by the numerical solutions of various nonlocal variational problems %with
parameterized by a finite range of interactions, which also serve as an illustrative application of the general theory developed here. In recent years, there has been much interest in the study of such variational problems which chiefly employ nonlocal integral operators so that they may retain more general physical properties and allow for more singular solutions. Nonlocal models with a finite range of interactions, denoted by a constant $\delta$, can also serve as a bridge connecting discrete and classical local models, a connection made apparent by considering either discrete approximations or the localized limit as $\delta\to 0$ \cite{Du2019book}.
	
	Due to the need to handle more complex nonlocal interactions and possibly singular solutions \cite{Silling2000,bobaru2016handbook,Du2019book,diehl2019review,foss2022convergence,silling2017stability,jha2019finite,jha2021finite,leng2020asymptotically,trask2019asymptotically,yu2021asymptotically,you2020asymptotically}, developing reliable and robust numerical discretizations to simulate nonlocal models has been a challenging task.
	The abstract notion of an asymptotically compatible (AC) discretization was rigorously introduced in \cite{tian2014sinum} for a class of parameterized linear variational problems, motivated by the need to develop robust discretizations of the aforementioned nonlocal problems.
    %parametrized by the horizon parameter $\delta$. 
    Indeed, as illustrated in  \cite{tian2013sinum}, discretizations that are not AC may risk converging to a physically inconsistent limit as $\delta$ and $h$ are refined simultaneously. The framework was further extended in \cite{tian2020sirv} to allow for non-adjoint problems and rough data.
	
    Our objectives are two-fold: First, we extend the AC framework to a wide class of parametrized nonlinear variational problems in the context of gamma-convergence. The details are given in \Cref{sec:ac}. Mathematically speaking, this framework is not confined to nonlocal problems but is broadly applicable, which offers a significant and novel contribution on its own. Second, given the motivation discussed above, we %apply
    present illustrative applications of
    the AC framework to two examples of nonlinear nonlocal problems. The first example, presented in 
         \Cref{sec:example},
    involves heterogeneous localization of the horizon parameter, and the second example, studied in 
         \Cref{sec:varying},
    is posed on a domain that changes with the horizon parameter. Neither of these examples have been treated previously even in the linear case. 
    %On the first, we build on the notion of Gamma-convergence of functionals to treat a wide class of nonlinear problems. 
    %On the latter, 
    The nonlocal problems in the first example are characterized by the spatial variation of the nonlocal horizon parameter over the domain. In particular, to impose local boundary conditions for the nonlocal models defined in a bounded domain, the localization on the domain boundary means that the parameter can vary from a finite positive value, and tends to zero on the boundary. Thus, AC schemes are ideally suited for reliable and robust numerical approximations. 
    In addition, to better relate to the earlier abstract frameworks developed in \cite{tian2014asymptotically,tian2020sirv}, we also consider a second example in which the horizon parameter takes on a constant value over the spatial domain, but in which the domain depends on the said parameter.
    To further demonstrate the broad applicability of the framework, for the two application examples, we choose to work with different types of boundary conditions, namely, natural Neumann type local boundary conditions for the former case, and Dirichlet nonlocal conditions for the latter case.
    %we also specialize our discussion to the linear setting.
    %and to compare and contrast the assumptions and conclusions.}
	
	% We note that the study of AC schemes are also relevant for the numerical discretization of PDEs, the study of limiting behaviors of the harmonic extensions of point clouds and discrete graph Laplacians \cite{belkin2009constructing,burago2015graph,calder2022improved,calder2023rates,li2017point,garcia2020error,shi2021convergence}, and ...

    Our study here of nonlocal models with heterogeneous localization can also be considered a continuation of 
    %the earlier 
    recent effort made in \cite{scott2023nonlocal,scott2023nonlocal2}, which were motivated by earlier studies in \cite{tian2017trace,du2022fractional}. These models provide effective ways to seamlessly couple with local models to allow for more effective modeling and simulations \cite{tao2019nonlocal}.

	\section{Asymptotically compatible discretization}\label{sec:ac}
	We first present the main ingredients of 
	the abstract framework of AC schemes, which have been first
 considered for parameterized linear problems, see \cite{tian2014sinum,tian2020sirv}. 
	The generalized framework is energy-based, and can handle more general nonlinear problems, as well as allow for the case of rough data as in \cite{tian2020sirv}.	Our presentation starts with the description of the parametrized problems and associated Banach spaces. We then formulate the finite dimensional approximations. The gamma-convergence results in various regimes are then considered in the framework of AC schemes, followed by discussions on the convergence of the minimizers.
 
	%We make some slight reorganization of the main assumptions and conclusions presented in \cite{tian2020sirv} where the framework is established for weak formulations of possibly non-self-adjoint linear problems associated with two families of reflexive Banach spaces $\{\cT_\si, \si \in [0, \infty]\}$ serving as trial solution spaces, and $\{\cX_\si, \si \in [0, \infty]\}$, serving as test function spaces. As we focus on self-adjoint problems associated with energy minimizations in this work, we present the special case of $\{\cX_\si=\cT_\si\}$, which was considered originally in \cite{tian2014sinum}, while allowing for the case of rough data as in  \cite{tian2020sirv}.
	
	\subsection{Parametrized problems}
	The trial spaces for the associated minimization problems are reflexive Banach spaces $\{\cT_\si\}$
 %The family of spaces to be adopted are 
 parameterized by  $\si\in \R\cup\{\pm \infty\}$ with corresponding norms $\{ \| \cdot \|_{\cT_{\si}} \}$, with associated dual spaces denoted $\cT_{-\si} := \cT_{\si}^\ast$.
	In prior applications of the AC framework  to variational problems \cite{tian2014sinum,tian2020sirv},
	$\cT_{0}$ is taken to be a Banach space; a typical example is $\cT_0 = L^p(\Omega)$ for a fixed exponent $p \in (1,\infty)$, where $\Omega \subset \bbR^d$ is a bounded domain. 
	Meanwhile, in the linear setting $\cT_{0}$ serves as the pivot space between $\cT_{-\si}$ and $\cT_{\si}$,  so that a distribution extended to $\cT_\sigma$ can be realized via the duality pairing $\langle \cdot , \cdot \rangle$ given by the inner product on $\cT_0$.
	%between $\cT_{-\si}$ and $\cT_{\si}$ is given as the extension of the inner product on $\cT_{0}$. 
	
	Following \cite{tian2014sinum,tian2020sirv}, we state the following assumptions.
	
	\begin{assumption}\label{assumption1}\rm
		The spaces $\{\cT_\si\}$ satisfy the following:
		\begin{enumerate}
			\item[\rm (i)] \textit{Uniform embedding} property: there are positive constants $M_1$ and $M_2$, independent of $\si \in(0, \infty]$, such that
			\[
			{M_1} \| u\|_{\cT_0}\leq  \| u\|_{\cT_{\si}}\quad \forall u\in \cT_{\si}.
			%\quad
			%\text{and}\quad
			%\| u\|_{\cT_{\si}} \leq M_2  \| u\|_{\cT_\infty} %\quad \forall u\in \cT_{\infty}.
			\]
			
			\item[\rm (ii)] {\it Asymptotically compact embedding} property for $\{\cT_\si\}$: for all sequences $\{v_\si \in \cT_{\si}\}$, if there is a constant $C>0$ independent of $\si$ such that
			$ \| v_\si\|_{\cT_{\si}} \leq C$  $\forall \si$, 
			then the sequence $\{v_\si\}$ is relatively compact in $\cT_0$ and each limit point is in~$\cT_\infty$.
		\end{enumerate}
	\end{assumption}
	
	Now, we consider some assumptions on parametrized energy functionals that make the latter amenable to direct methods in the calculus of variations \cite{giusti2003direct}.
	%These are simplified in the self-adjoint case. In the more general non-self-adjoint case, inf-sup conditions are imposed.
	%The principal part of the energies, $\cG_\sigma$, satisfies its own version of the continuity and coercivity analogous to \Cref{assumption2}:
	\begin{assumption}\label{assumption2}
		Let $\cE_\sigma : \cT_\sigma \to [0,\infty)$, $\sigma \in (0,\infty]$ be a %convex 
		$\cT_\sigma$-weakly lower semicontinuous
		functional that satisfies the following:
		\begin{enumerate}
			\item[\rm (i)] Continuity:
			% there exists an exponent $p \in (1,\infty)$ and a constant $C_2$ independent of $\si$ such that
			there exists a function $\varphi \in C^0([0,\infty)^2)$ such that
			\[
			|\cE_\si(u) - \cE_\si(v)| \leq 
			%C_2 (1 + \|u\|_{\cT_{\si}}^{p-1} + \|v\|_{\cT_{\si}}^{p-1} )
			\varphi(\vnorm{u}_{\cT_\si}, \vnorm{v}_{\cT_\si}) \|u-v \|_{\cT_{\si}} ,\quad \forall u, v \in \cT_{\si}.
			\]
			\item[\rm (ii)]  Coercivity: there exists an exponent $p \in (1,\infty)$, 
			%there exists 
			a constant $\alpha > 0$ independent of $\si$, and a function $\psi \in C^0([0,\infty))$ with $\lim\limits_{t \to \infty} \frac{\psi(t)}{t^{p}} = 0$ such that
			\[ 
			\alpha \vnorm{u}_{\cT_\si}^p \leq \cE_\si(u) + \psi( \vnorm{u}_{\cT_\si} ), \quad \forall u\in \cT_{\si}.
			\]
		\end{enumerate}
	\end{assumption}
	
	We consider the following family of parametrized minimization problems:
	%given $f_\si \in \cT_{-\si}$,
	\beq \label{eq:minprob}
	\text{find $u_\si \in \cT_\si$ such that}\; \cE_{\sigma}(u_{\sigma}) = \min_{ v \in \cT_\sigma } \cE_{\sigma}(v),
	%, \quad \text{ where }     \cE_\sigma(u) := \cG_{\sigma}(u) - \vint{f_{\sigma},u},
	\eeq
	for any value of the parameter $\si\in (0, \infty]$. 
	
	We assume the following convergence condition on the energies in the asymptotic regime $\sigma \to \infty$:
	\begin{assumption}\label{assumption3}
		The functionals $\{ \overline{\cE}_\sigma \}_\sigma$ $\Gamma$-converge as $\sigma \to \infty$ in the $\cT_0$-strong topology to $\overline{\cE}_\infty$, where $\overline{\cE}_\sigma$ 
		%and $\overline{\cE}_\infty$ 
		for $\sigma \in (0,\infty]$ are the extensions defined by
		\begin{equation*}
			\overline{\cE}_\sigma(u) = 
			\begin{cases}
				\cE_\sigma(u) &\text{ for } u \in \cT_\sigma\,, \\
				+\infty &\text{ for } u \in \cT_0 \setminus \cT_\sigma\,.
			\end{cases}
			% \quad \text{ and } \quad 
			% \overline{\cE}_\infty(u) = 
			%     \begin{cases}
				%         \cE_\infty(u) &\text{ for } u \in \cT_\infty\,, \\
				%         +\infty &\text{ for } u \in \cT_0 \setminus \cT_\infty\,.
				%     \end{cases}
		\end{equation*}
	\end{assumption}
For ease of reference, we here recall the definition of $\Gamma$-convergence; see also \cite{braides2002gamma}.
\begin{defn}
	Let $X$ be a metric space. We say that a sequence of functionals $\{ I_\sigma : X \to \overline{\bbR} \}_\sigma$ $\Gamma$-converges as $\sigma \to \infty$ in $X$ to $I_\infty$ if the following properties hold:
	\begin{enumerate}
		\item[\rm (1)] The liminf inequality: For each $v \in X$ and for any sequence $\{ v_\si \}_\si$ that converges to $v$ in $X$,
		\[
		I_\infty(v) \leq \liminf_{\si\to\infty} I_\sigma(v_\sigma).
		\]
		\item[\rm (2)] Existence of a recovery sequence: For each $v\in X$, there exists $\{ v_\sigma \}_{\sigma}$ that converges to $v$ in $X$ such that 
		\[
		\limsup_{\sigma\to\infty} I_\sigma(v_\sigma) \leq    I_\infty(v). 
		\]
	\end{enumerate}
\end{defn}
%\begin{defn}
%We say that the sequence of functionals $\{ \overline{\cE}_\sigma \}_\sigma$ $\Gamma$-converge as $\sigma \to \infty$ in the $\cT_0$-strong/weak topology to $\overline{\cE}_\infty$ if the following properties hold:
%\begin{enumerate}
%    \item The liminf inequality: Assume $v_\si$ converges to $v$ strongly/weakly in $\cT_0$, then
%    \[
%    \overline{\cE}_\infty(v) \leq \liminf_{\si\to\infty} \overline{\cE}_\sigma(v_\sigma).
%    \]
%    \item Existence of a recovery sequence: For each $v\in \cT_0$, there exists $\{ v_\sigma \}_{\sigma>0}$ where $v_\sigma$ converges to $v$ strongly/weakly in $\cT_0$ such that 
%    \[
%    \limsup_{\sigma\to\infty} \overline{\cE}_\sigma(v_\sigma) \leq    \overline{\cE}_\infty(v). 
%    \]
%\end{enumerate}
%\end{defn}

	%By all assumptions made in the above, we have the existence and uniqueness of $u_\si$ and the convergence of the asymptotic limit \cite{tian2014sinum}:
	\begin{theorem}[Convergence of minimizers as $\si\to\infty$]
		\label{theorem:convergence1}
		Given \Cref{assumption1}, \Cref{assumption2}, and \Cref{assumption3}, there exists a $u_\sigma \in \cT_\sigma$ satisfying \eqref{eq:minprob} for $\sigma \in (0,\infty)$. Furthermore, there exists a subsequence of $\{u_\si\}_\si$ (not relabeled) and there exists $u_\infty \in \cT_\infty$ such that
		\beq
		\label{eq:gammaconvergence1}
		\| u_{\si}-u_{\infty}\|_{\cT_{0}} \to  0\quad \text{as } \si\to\infty,
		\eeq
        and $u_\infty$ satisfies \eqref{eq:minprob} with $\si=\infty$.
		If additionally $\cE_\si$ is convex for $\si \in (0,\infty]$ then the minimizers $u_\si$  are unique.
	\end{theorem}
	
	\begin{proof}
		Let $\{ u_{\sigma}^n \}_n \subset \cT_\sigma$ be a minimizing sequence of $\cE_\sigma$. Then by \Cref{assumption2}(ii) 
		%and \Cref{assumptionf}(i)
		% \begin{equation*}
			%     C \geq \cE_\sigma(u_\sigma^n) \geq \psi( \vnorm{u_\sigma^n}_{\cT_\sigma}) 
			%     %\geq \alpha \vnorm{u_\sigma^n}_{\cT_\sigma}^p - C_3 \Vnorm{u_\sigma^n}_{\cT_\sigma}
			% \end{equation*}
		%for a constant $C$ independent of $\sigma$. It 
		it follows from the property of $\psi$ that the sequence $\{ \vnorm{ u_{\sigma}^n}_{\cT_\sigma} \}_n$ is uniformly bounded, and so there exists $u_\sigma \in \cT_\sigma$ such that $u_\sigma^n \rightharpoonup u_\sigma$ weakly in $\cT_\sigma$. Since $\cE_\sigma$ 
		%is convex and continuous on $\cT_\sigma$ by \Cref{assumption2}(i), it 
		is $\cT_\sigma$-weakly lower semicontinuous,
		%and so 
		$\cE_\sigma(u_\sigma) \leq \liminf_{n \to \infty} \cE_\sigma(u_\sigma^n) = \min_{v \in \cT_\sigma} \cE_\sigma(v)$.
		%The minimizer $u_\sigma$ is unique thanks to the convexity of $\cE_\sigma$.
		
		The convergence \eqref{eq:gammaconvergence1} of minimizers follows from applying the $\Gamma$-convergence \Cref{assumption3} and the compact embedding property \Cref{assumption1}(ii) inside the convergence framework described in \cite[Theorem 1.21]{braides2002gamma}. 
	\end{proof}

	\subsection{Finite dimensional approximations}
	Finite-dimensional approximation of parameterized problems is a popular subject of numerical analysis, see \cite{arnold1997locking,brezzi1980finite,quarteroni2007numerical,guermond2010asymptotic,gunzburger1996finite} and additional references cited in \cite{tian2020sirv}.
	
	We now focus on finite-dimensional Galerkin approximations of problems \eqref{eq:minprob} associated to the energies $\{\cE_\si\}$ for $\si\in(0,\infty]$.
    Consider a family of closed subspaces $\{W_{\si, h}\subset \cT_{\si}\}$ parametrized by an additional real parameter $h\in(0, h_0]$ that measures the level of resolution (or the reciprocal of the number of degrees of freedom).
    We assume that the dimension of $W_{\si, h}$  is finite for all $\si \in(0,\infty]$ and $h \in(0, h_0]$, 
    %$\dim W_{\si, h} <\infty$.
    and we further assume that
	$$Z_h=
	\dim  \bigcup_{\si\in (0,\infty]} W_{\si, h} <\infty. $$
	With $W_{\si, h} \subset \cT_{\si}$, we are effectively adopting a standard internal or 
 %equivalently 
 conforming 
 %type, 
 Galerkin approximation approach \cite{tian2014sinum,tian2020sirv}.
	
	The approximation to the problem \eqref{eq:minprob}
        %to $u_\si$ 
        in the subspace $W_{\si,h}$ is %defined via
	\begin{equation}\label{eq:minprobnum}
		\begin{gathered}
			\text{find } u_{\sigma,h} \in W_{\sigma,h} \text{ such that } \cE_{\sigma}(u_{\sigma,h}) = \min_{ v \in W_{\sigma,h} } \cE_{\sigma}(v).
		\end{gathered}
	\end{equation}
	
	We need some basic assumptions on the approximation
	properties as stated below, adapted to our case and in line with \cite{tian2014sinum}.
	The first assumption ensures
	the convergence of $W_{\si, h}$
	to $\cT_\si$ as $h \to 0$ for each $\si$, and the second assumption 
	is concerned with the limiting behavior as $\si\to\infty$ for different sequences of $h$.

	\begin{assumption}\label{assumption4}\rm
		For the family of finite dimensional subspaces  $\{W_{\si, h}\subset \cT_\si\}$
		parametrized by $\si\in(0, \infty]$ and $h\in(0, h_0]$, the following properties hold:
		\begin{enumerate}
			%\item[\rm (i)] Discrete inf-sup: there exists a constant $\tilde\alpha>0$ independent of $\si$ such that
			%\[
			%\inf_{u\in W_{\si,h}} \sup_{v\in X_{\si,h}} \frac{ a_\si(u, v) }{\| u\|_{\cT_\si} \|v \|_{\cT_\si}}\geq \tilde\alpha >0 .
			%\]
			\item[\rm (i)] Approximation property: for each $\si\in(0, \infty]$, the family $\{W_{\si, h}, h\in(0, h_0]\}$
			is dense in $\cT_{\si}$. That is,  $\forall u\in\cT_\si$,\
			$\exists$ a sequence $\{u_n\in W_{\si, h_n}\}$ with $h_n\to0$ as $n\to\infty$ such that
			\beq\label{eq:finitedense}
			\displayindent0pt
			\displaywidth\textwidth
			\| u-u_n\|_{\cT_\si}\to 0 \quad \text{as} \quad n\to\infty.
			\eeq
			
			\item[\rm (ii)]
			Limit of approximate spaces: for any given $h>0$,
			\beq
			\label{eq:disubsp}
			\displayindent0pt
			\displaywidth\textwidth
			W_{\infty,h}=\cT_\infty\cap \left(\bigcup_{\sigma \in (0,\infty]} W_{\si,h}\right).
			\eeq
                \item[\rm (iii)] The spaces $\{ W_{\si, h}, \si\in (0,\infty), h\in(0, h_0]\}$  are {\it asymptotically approximating}
                for the minimization problems \eqref{eq:minprob}: %in the following sense.
			for any $v\in W_{\infty,h}$ (with the convention $W_{\infty,0} = \cT_\infty$) and any sequence $\{ (\sigma_n,h_n) \}$ satisfying
            \begin{equation*}
                \begin{cases}
                    \lim\limits_{n \to \infty} \sigma_n = \infty \text{ and } h_n = h \; \forall n & \text{ if } h > 0\,, \\
                    \lim\limits_{n \to \infty} \sigma_n = \infty \text{ and } \lim\limits_{n \to \infty} h_n = 0 & \text{ if } h = 0\,,
                \end{cases}
            \end{equation*}
            %$\lim\limits_{n \to \infty}(\sigma_n,h_n)$ $\sigma_n \to \infty$, $h_n \to h_\infty \in [0,h_0]$ for $n \to \infty$, 
           %\tcb
           {we assume the existence of} % assume that there exists
           a recovery sequence $\{v_{\sigma_n}\}_{n} \subset \cT_0$ for the $\Gamma$-limit described in \Cref{assumption3}, i.e. 
			\begin{equation*}
				\lim\limits_{n \to \infty}\vnorm{v_{\sigma_n} - v}_{\cT_0} = 0 \quad \text{ and } \quad \cE_\infty(v) = \lim\limits_{n \to \infty} \cE_{\sigma_n}(v_{\sigma_n}),
			\end{equation*}
			%with the property: There exists 
   %\tcb{
   together with the existence of
   %}
   a sequence $\{v_n\in W_{\si_n, h_n}\}_{n \to \infty}$ that approximates $\{ v_{\si_n}\in \cT_{\si_n} \}$ in the sense that
   $\| v_n -v_{\si_n} \|_{\cT_{\si_n}}\to 0$ as $n\to\infty$. 
		\end{enumerate}
	\end{assumption}

    \begin{remark}\label{rmk:assumption4'}
       In specific cases and applications, the direct verification of \Cref{assumption4}(iii) may not be straightforward.
         The convergence results in the example in the next section will be established under the following more restrictive assumption in place of \Cref{assumption4}, for ease of the verification. 
         %in examples.
        
        \begin{assumption}\label{assumption4'}
        For the family of spaces $\cT_\si$ and finite dimensional subspaces  $\{W_{\si, h}\subset \cT_\si\}$
		parametrized by $\si\in(0, \infty]$ and $h\in(0, h_0]$, the following properties hold:
        \begin{enumerate}
            \item[\rm (i)] $\cT_\infty \subset \cT_\sigma$ for all $\sigma < \infty$ with a constant $M_2$ independent of $\si$ such that
            $$\vnorm{u}_{\cT_\si} \leq M_2 \vnorm{u}_{\cT_\infty}, \quad \forall u\in \cT_\infty.$$
            \item[\rm (ii)] For each $h \in (0,h_0]$, $W_{\infty,h} = \cT_\infty \cap W_{\si,h}$ for any $\si \in (0,\infty]$.
        
            \item[\rm (iii)] Approximation property: 
            %for each $\si\in(0, \infty]$, the family $\{W_{\si, h}, h\in(0, h_0]\}$ is dense in $\cT_{\si}$. That is,  
            $\forall u\in\cT_\si$ for $\sigma\in (0,\infty]$,\
			$\exists$ a sequence $\{u_n\in W_{\si, h_n}\}$ with $h_n\to0$ as $n\to\infty$ such that
			%\beq
			%\displayindent0pt
			%\displaywidth\textwidth
                $
			\| u-u_n\|_{\cT_\si}\to 0 \text{ as } n\to\infty.
                $
			%\eeq
   
            \item[\rm (iv)] The following ``pointwise limit'' holds:
            \begin{equation*}
            \cE_\infty(v) = \lim\limits_{n \to \infty} \cE_{\sigma_n}(v) \quad \forall v \in \cT_\infty.
        \end{equation*}

            \item[\rm(v)] The spaces $W_{\si,h}$ are \textit{asymptotically dense} in $\cT_\infty$, i.e.
        \begin{equation*}
            \forall v \in \cT_\infty \; \exists \{v_n \in W_{h_n,\si_n} \}_{h_n \to 0, \si_n \to \infty} \text{ such that } \lim\limits_{n \to \infty} \vnorm{ v_n - v }_{\cT_\infty} = 0.
        \end{equation*}
        \end{enumerate}
        
        \end{assumption}

    \end{remark}

    Note that \Cref{assumption4'}(iv) and \Cref{assumption4'}(v) only make sense because of \Cref{assumption4'}(i).
    
    \begin{lemma}
        \Cref{assumption4'} $\Rightarrow$ \Cref{assumption4}.
    \end{lemma}
    \begin{proof}
        First, \Cref{assumption4}(i) is a restatement of \Cref{assumption4'}(iii). Next, the equality \Cref{assumption4}(ii) follows from \Cref{assumption4'}(ii). 

        Finally, we show both cases of \Cref{assumption4}(iii). Let $h > 0$. Then we choose $\{ v_{\si_n} \}_n = v$ for all $n$, i.e. we choose the constant recovery sequence, which is possible thanks to \Cref{assumption4'}(iv). Then, by the description of $W_{\infty,h}$ in \Cref{assumption4'}(ii), we can choose the sequence $\{ v_n \}_n = v$ for all $n$, which trivially satisfies the asymptotic approximation property for the constant recovery sequence. %We also use the uniform embedding $\cT_\infty \subset \cT_{\si}$.

        Now, let $h = 0$. Then we again choose the the constant recovery sequence $v_{\si_n} = v$ for all $n$. Then we choose the sequence $\{v_n\}$ using the asymptotic density property \Cref{assumption4'}(v) to get $\vnorm{v-v_n}_{\cT_{\si_n}} \leq M_2 \vnorm{v-v_n}_{\cT_\infty} \to 0$.
    \end{proof}

	\subsection{Gamma-convergence results}
	
	Before we show the main convergence results in the next section, we need some auxiliary results on the $\Gamma$-limits of the functionals $\cE_\sigma$ when restricted to finite-dimensional subspaces. We collect these in the following theorems in this section.

    \begin{figure}[htbp]
        \centering
        \begin{tikzpicture}[scale=0.9]
			\tikzset{to/.style={->,>=stealth',line width=.8pt}}
			\node(v1) at (0,2.5) {\textcolor{red}{$\cE_{\si, h}$}};
			\node (v2) at (5.,2.5) {\textcolor{red}{$\cE_{\infty, h}$}};
			\node (v3) at (0,0) {\textcolor{red}{$\overline{\cE}_\si$}};
			\node (v4) at (5.,0) {\textcolor{red}{$\overline{\cE}_\infty$}};
			\draw[to] (v1.east) -- node[midway,above] {\footnotesize{\textcolor{blue}{\sc\Cref{theorem:gammaconvergence:sigmatoinf}}}}  node[midway,below] {\footnotesize{\textcolor{blue}{$\si\to\infty$}}}
			(v2.west);
			\draw[to] (v1.south) -- node[midway,left] {\footnotesize{\textcolor{blue}{\sc  \Cref{theorem:gammalimit:hto0}}}} node[midway,right] {\footnotesize{\textcolor{blue}{$h\to0$}}} (v3.north);
			\draw[to] (v3.east) -- node[midway,below] {\footnotesize{\textcolor{blue}{\sc \Cref{assumption3}}}} node[midway,above] {\footnotesize{\textcolor{blue}{$\si\to\infty$}}} (v4.west);
			\draw[to] (v2.south) -- node[midway,right] {\footnotesize{\textcolor{blue}{\sc \Cref{theorem:gammalimit:hto0}}}} node[midway,left] {\footnotesize{\textcolor{blue}{$h\to0$}}}(v4.north);
			\draw[to] (v1.south east) to[out = 2, in = 180, looseness = 1.2] node[midway,yshift=-0.05cm] 
            %{\footnotesize\hbox{\shortstack[l]{ {\textcolor{blue}{\sc $\,h\to0\;\;\;$\Cref{theorem:gammaconvergence:ac}}}\\{\textcolor{blue}{$\;\si\to\infty$}} }}} 
            {\footnotesize\hbox{\shortstack[c]{ {\textcolor{blue}{\hspace{-4pt}\sc$h\to0$}}\\{\textcolor{blue}{$\si\to\infty$}}\\{\textcolor{blue}{\sc\Cref{theorem:gammaconvergence:ac}} }}}}
            (v4.north west);
		\end{tikzpicture}
		\vskip-10pt
		\caption{A diagram of the gamma-convergence results for the functionals.}\label{fig:diagram2}
        \end{figure}
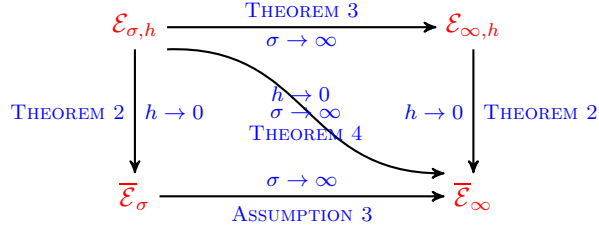

	To fix ideas, for any $\sigma \in (0,\infty]$ we define the functionals $\cE_{\sigma,h}$ to be the restrictions of $\cE_\sigma$ to the subspaces $W_{\sigma,h}$ that remain compatible with the framework of $\Gamma$-convergence, i.e.
	\begin{equation*}
		\cE_{\sigma,h}(v) := 
		\begin{cases}
			\cE_\sigma(v) &\text{ for } v \in W_{\sigma,h}\,, \\
			+\infty &\text{ for } v \in \cT_0 \setminus W_{\sigma,h}\,.
		\end{cases}
	\end{equation*}
	With this definition, \eqref{eq:minprobnum} can be restated as
	\begin{equation*}
		\text{ find } u_{\sigma,h} \in W_{\sigma,h} \text{ such that } \cE_{\sigma,h}(u_{\sigma,h}) = \min_{ v \in W_{\sigma,h} } \cE_{\sigma,h}(v).
	\end{equation*}

	\begin{theorem}[Convergence as $h \to 0$ with fixed $\sigma$]\label{theorem:gammalimit:hto0}
		Under all previous assumptions, the functionals $\{\cE_{\sigma,h} \}_{h}$ $\Gamma$-converge as $h \to 0$ in the $\cT_\sigma$-strong topology to $\overline{\cE}_\sigma$.
		%$\overline{\cE}_{\sigma,0}$, where $\overline{\cE}_{\sigma,0} = \cE_{\sigma}$ on $\cT_\sigma$ and $+\infty$ otherwise.
	\end{theorem}
	
	\begin{proof}
		First we prove the liminf inequality.
		Let $v_h \to v_0$ strongly in $\cT_\sigma$. If $\liminf_{h \to 0} \cE_{\sigma,h}(v_h) = \infty$, then there is nothing to show, and so we can assume that $\liminf_{h \to 0} \cE_{\sigma,h}(v_h) < \infty$. Thus $v_h \in W_{\sigma,h}$ and $\cE_{\sigma,h}(v_h) = \cE_\sigma(v_h)$ for all $h > 0$. 
		Since $\cE_\sigma$ is $\cT_\sigma$-lower semicontinuous, it follows that $\cE_\sigma(v_0) \leq \liminf_{h \to 0} 
        %\overline{\cE}_{\sigma,h}(v_h)
        \cE_{\sigma,h}(v_h)$. 
		
		To see the limsup inequality, we choose, for any $v \in \cT_\si$, a sequence $\{v_h\}_h$ that satisfies \Cref{assumption4}(i); this serves as a recovery sequence. Indeed, applying \Cref{assumption2}(i),
        \begin{equation*}
            \begin{split}
                |\cE_{\si,h}(v_h) - \overline{\cE}_\si(v)|
                &= |\cE_{\si}(v_h) - \cE_\si(v)| \\
                &\leq \varphi( \vnorm{v_h}_{\cT_\si}, \vnorm{v}_{\cT_\si} ) \vnorm{v_h - v}_{\cT_\si} \to 0 \text{ as } h \to 0.
            \end{split}
        \end{equation*}

		% by \Cref{assumption2}(i) and continuity of $f_\sigma$,
		% \begin{equation*}
			%     |\cG_\sigma(v) - \cG_\sigma(v_h)| \leq C ( \vnorm{f_\sigma}_{\cT_\sigma^*} + \vnorm{v}_{\cT_\sigma}^{p-1} + \vnorm{v_h}_{\cT_\sigma}^{p-1}) \vnorm{v-v_h}_{\cT_\sigma} \to 0 \text{ as } h \to \infty.
			% \end{equation*}
	\end{proof}
	
	\begin{corollary}\label{cor:gammalimit:hto0:weak}
		Under all of the previous assumptions, the functionals $\{\cE_{\sigma,h} \}_{h}$ $\Gamma$-converge as $h \to 0$ in the $\cT_\sigma$-weak topology to $\overline{\cE}_\sigma$.
	\end{corollary}
	
	\begin{proof}
		The result follows from the same methods as the previous theorem, using that $\cE_\sigma$ is $\cT_\sigma$-weakly lower semicontinuous as noted in the proof of \Cref{theorem:convergence1}.
	\end{proof}
	
	\begin{theorem}[Convergence as $\sigma \to \infty$ with fixed  $h$]\label{theorem:gammaconvergence:sigmatoinf}
		Let $h \in (0,h_0]$. Then the functionals $\{\cE_{\sigma,h} \}_h$ $\Gamma$-converge as $\sigma \to \infty$ in the $\cT_0$-strong topology to $\cE_{\infty,h}$.
	\end{theorem}
	
	\begin{proof}
		To prove the liminf inequality, we can assume similarly to the previous theorem that $\liminf_{\sigma \to \infty} \cE_{\sigma,h}(v_{\sigma,h}) < \infty$ for a sequence $\{ v_{\sigma,h} \}_\sigma$ that converges to $v_{\infty}$ strongly in $\cT_0$.    
		Thus $v_{\sigma,h} \in W_{\sigma,h} \subset \cT_\sigma$ and $\cE_{\sigma,h}(v_{\sigma,h}) = \cE_\sigma(v_{\sigma,h})$ for all $\sigma > 0$.
		Additionally, $v_{\infty,h} \in \cT_\infty$ by \Cref{assumption1}(ii) and \Cref{assumption2}(ii).
		Since $\overline{\cE}_{\sigma}$ $\Gamma$-converges to $\overline{\cE}_{\infty}$ in the $\cT_0$-strong topology by \Cref{assumption3}, it follows that $\overline{\cE}_{\infty}(v_\infty) = \cE_{\infty}(v_\infty) \leq \liminf_{\sigma \to \infty} \cE_{\sigma,h}(v_{\sigma,h})$.
		We just need to show that  $\cE_\infty(v_\infty) = \cE_{\infty,h}(v_\infty)$, i.e. we need to show that $v_\infty \in W_{\infty,h}$.
		But since $v_{\sigma,h} \in W_{\sigma,h} \subset \bigcup_{\sigma \in (0,\infty]} W_{\sigma,h}$ by \Cref{assumption4}(ii), and the dimension of this space, $Z_h$, remains finite, it follows that $v_{\infty} \in \bigcup_{\sigma \in (0,\infty]} W_{\sigma,h}$ as well.
        Thus $v_\infty \in W_{\infty,h}$, and the liminf inequality holds.

		For the limsup inequality, we first note that if $\cE_{\infty,h}(v) < \infty$ then $v \in W_{\infty,h}$ with $\cE_{\infty,h}(v) = \cE_{\infty}(v)$.
		Let $\{v_{\sigma} \}_\sigma \subset \cT_\sigma$ be a recovery sequence for the continuum problem, guaranteed to exist by \Cref{assumption3}, so that $\lim\limits_{\sigma \to \infty} \cE_{\sigma}(v_\sigma) = \cE_\infty(v)$ and $\lim\limits_{\sigma \to \infty} \vnorm{v_\sigma - v}_{\cT_0} = 0$.
		Then in the context of \Cref{assumption4}(iii) with $h_n = h$ for all $n \in \bbN$, we can obtain a sequence $\{ w_\si \in W_{\si,h} \}_{\si}$ with
		$\lim\limits_{\si \to \infty} \vnorm{w_\si - v_\si}_{\cT_\si} =0$.
		Therefore $\lim\limits_{\si \to \infty} \vnorm{ w_\si - v}_{\cT_0} = 0$ by \Cref{assumption1}, and we can use $\{w_\si \in W_{\si,h} \}_\si$ as a recovery sequence by \Cref{assumption2}(i):
		\begin{equation*}
			\begin{split}
				|\cE_{\infty,h}(v) - \cE_{\si,h}(w_\si)| &= |\cE_{\infty}(v) - \cE_\si(w_\si) \pm \cE_{\si}(v_\si)| \\
				&\leq |\cE_\si(w_\si) - \cE_{\si}(v_\si)| \\
				&\qquad 
				%+ C(1+ \vnorm{v_{\si}}_{\cT_\si}^{p-1} + \vnorm{w_{\si}}_{\cT_\si}^{p-1}) 
				\varphi( \vnorm{v_{\si}}_{\cT_\si}, \vnorm{w_{\si}}_{\cT_\si} ) \vnorm{v_{\si} - w_{\si} }_{\cT_\si} \to 0.
			\end{split}
		\end{equation*}

	\end{proof}
	
	\begin{theorem}[Convergence as $\sigma \to \infty$ and $h\to 0$]\label{theorem:gammaconvergence:ac} 
		Let $\sigma_n \to \infty$ and $h_n \to 0$ as $n \to \infty$. Then the functionals $\cE_{\sigma_n,h_n}$ $\Gamma$-converge to $\overline{\cE}_\infty$ in the $\cT_0$-strong topology as $n \to \infty$.
	\end{theorem}
	
	\begin{proof}
        %As noted above, we prove the result with \Cref{assumption4} in place of \Cref{assumption4'}.
		First we prove the liminf inequality.
		Let $v_n \subset W_{\sigma_n,h_n}$ converge to $v$ strongly in $\cT_0$. If $\liminf_{n \to \infty} \cE_{\sigma_n,h_n}(v_{n}) = \infty$, then there is nothing to prove, and so we can assume that $\liminf_{n \to \infty} \cE_{\sigma_n,h_n}(v_{n}) = \liminf_{n \to \infty} \cE_{\sigma_n}(v_{n}) < \infty$.
		But this is equal to $\liminf_{\sigma_n \to \infty} \overline{\cE}_{\sigma_n}(v_n)$, and since $\overline{\cE}_{\sigma}$ $\Gamma$-converges to $\overline{\cE}_{\infty}$ in the $\cT_0$-strong topology by \Cref{assumption3}, it follows that $\overline{\cE}_{\infty}(v) = \cE_{\infty}(v) \leq \liminf_{\sigma_n \to \infty} \overline{\cE}_{\sigma_n}(v_n) \leq \liminf_{n \to \infty} \cE_{\sigma_n,h_n}(v_n)$.
		
		For the limsup inequality, we first note that if $\overline{\cE}_{\infty}(v) < \infty$ then $v \in \cT_\infty$ with $\overline{\cE}_{\infty}(v) = \cE_{\infty}(v)$. Then using the asymptotic approximation for a continuum recovery sequence $\{v_{\sigma_n} \in \cT_{\sigma_n} \}$, we choose an asymptotically-compatible recovery sequence $\{v_n\}$ that satisfies \Cref{assumption4}(iii) to get 
		\begin{equation*}
			\begin{split}
				|\cE_{\sigma_n}(v_n) - \cE_\infty(v)| &\leq |\cE_{\sigma_n}(v_{\sigma_n}) - \cE_{\sigma_n}(v_n)| + |\cE_{\sigma_n}(v_{\sigma_n}) - \cE_\infty(v)| \\
				&\leq \varphi(\vnorm{ v_{\sigma_n} }_{\cT_{\sigma_n}}, \vnorm{v_n}_{\cT_{\sigma_n}}) \vnorm{ v_{\sigma_n} - v_n}_{\cT_{\sigma_n}} \\
				&\qquad + |\cE_{\sigma_n}(v_{\sigma_n}) - \cE_\infty(v)| \to 0 \text{ as } n \to \infty.
			\end{split}
		\end{equation*}
	\end{proof}

	\subsection{Convergence of finite dimensional approximations to minimizers}
 \label{sec:convergnumer}
	
	We show the asymptotic compatibility of discrete solutions to the problems \eqref{eq:minprob}.
	\Cref{prop:convergence} is the main result of the AC framework for minimization problems, which is also illustrated in \Cref{fig:diagram} as done similarly in \cite{tian2014sinum}.
 
	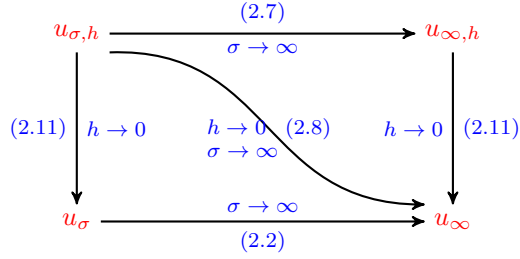
\begin{figure}[htbp]
		\centering
		\begin{tikzpicture}[scale=1]
			\tikzset{to/.style={->,>=stealth',line width=.8pt}}
			\node(v1) at (0,2.5) {\textcolor{red}{$u_{\si, h}$}};
			\node (v2) at (5.,2.5) {\textcolor{red}{$u_{\infty, h}$}};
			\node (v3) at (0,0) {\textcolor{red}{$u_\si$}};
			\node (v4) at (5.,0) {\textcolor{red}{$u_\infty$}};
			\draw[to] (v1.east) -- node[midway,above] {\footnotesize{\textcolor{blue}{\sc  \eqref{eq:gammaconvergence3}}}}  node[midway,below] {\footnotesize{\textcolor{blue}{$\si\to\infty$}}}
			(v2.west);
			\draw[to] (v1.south) -- node[midway,left] {\footnotesize{\textcolor{blue}{\sc  \eqref{eq:gammaconvergence2}}}} node[midway,right] {\footnotesize{\textcolor{blue}{$h\to0$}}} (v3.north);
			\draw[to] (v3.east) -- node[midway,below] {\footnotesize{\textcolor{blue}{\sc \eqref{eq:gammaconvergence1}}}} node[midway,above] {\footnotesize{\textcolor{blue}{$\si\to\infty$}}} (v4.west);
			\draw[to] (v2.south) -- node[midway,right] {\footnotesize{\textcolor{blue}{\sc \eqref{eq:gammaconvergence2}}}} node[midway,left] {\footnotesize{\textcolor{blue}{$h\to0$}}}(v4.north);
			\draw[to] (v1.south east) to[out = 2, in = 180, looseness = 1.2] node[midway,yshift=-0.15cm] {\footnotesize\hbox{\shortstack[l]{ {\textcolor{blue}{\sc $\,h\to0\;\;\;$\eqref{eq:gammaconvergence4}}}\\ {\textcolor{blue}{$\,\si\to\infty$}} }}} (v4.north west);
		\end{tikzpicture}
		\vskip-10pt
		\caption{A diagram for asymptotically compatible schemes and convergence results.}\label{fig:diagram}
	\end{figure}
        
	The compactness results can be obtained from the $\sigma$-asymptotic compactness of \Cref{assumption1}(ii). %see also \cite{tian2013sinum}. 
    Combined with the $\Gamma$-convergence results of the previous section, which are summarized in the diagram shown in
    \Cref{fig:diagram2},
    %along side of \Cref{fig:diagram} for easy reference, 
    we have the following:
	
    \begin{proposition}[Convergence of minimizers]\label{prop:convergence}
        Given \Cref{assumption1}, \Cref{assumption2}, \Cref{assumption3}, and \Cref{assumption4},
        assume additionally that $\cE_\si$ is convex, so that the functions $u_\si$ and $u_{\si,h}$ defined by
		\eqref{eq:minprob} and \eqref{eq:minprobnum} for any given $\si\in (0,\infty]$ and $h \in [0,h_0]$ are unique.
		We then have the following convergence results.
		\begin{itemize}
                
			\item[\rm (1)] 
                Convergence with a fixed $\si\in \mbox{$(0,\infty]$} $ as $h\to 0$: We have 
		      %\begin{equation}
                \begin{equation}
                \label{eq:gammaconvergence2:weak}
                \displayindent0pt
			\displaywidth\textwidth
				u_{\si,h} \rightharpoonup u_\si \text{ weakly in } \cT_\si.
			\end{equation}
			
			\item[\rm (2)] 
			Convergence of minimizers in approximation spaces with fixed $h>0$ as $\si\to\infty$:
			for minimizers
			$u_{\si,h}$ of
			\eqref{eq:minprobnum} with $\si\in (0, \infty]$, 
			\begin{equation}
			\label{eq:gammaconvergence3}
			\displayindent0pt
			\displaywidth\textwidth
			\| u_{\si,h}-u_{\infty,h}\|_{\cT_{0}}
			\to
			0\quad \text{as } \si\to\infty.
			\end{equation}
                \item[\rm (3)] Asymptotic compatibility of approximations
			$\{u_{\si,h}\}$ defined by \eqref{eq:minprobnum}
			with a minimizer
			$u_\infty$ defined by \eqref{eq:minprob} for $\si=\infty$:
			$\forall$ sequences $\si_n\to \infty
			$ and $h_n\to 0$, 
			\begin{equation}
                \label{eq:gammaconvergence4}
			\displayindent0pt
			\displaywidth\textwidth
			\| u_{\si_n, h_n}-u_\infty \|_{\cT_0}\to 0, \;\text{ as } \; \si_n\to \infty.
			\end{equation}
		\end{itemize}
	\end{proposition}
 
        \begin{remark}
            The convexity of $\cE_\si$ was assumed for ease of stating the convergence results. If $\cE_\si$ is not convex, then one can establish in a similar way the convergence of subsequences of local minimizers to local minimizers, but we do not include a rigorous statement here.
            %since the result can be established in a similar way.
        \end{remark}
	
	\begin{remark}\label{rmk:weakimpliesstr}
		If $\cE_\si(v)$ contains a term whose $1/p$-th power defines an equivalent norm on $\cT_\si$ (the simplest example is $\cE_\si(v) = [v]_{\cT_\si}^p + \vint{f,v}$ for $f \in \cT_{-\si}$) then, in the context of item 1), $u_{\si,h} \to u_{\si}$ as $h \to 0$ in the strong topology of $\cT_\si$.
	\end{remark}
	
	\begin{proof}[Proof of \Cref{prop:convergence}]
		First we prove \eqref{eq:gammaconvergence2:weak}. 
		Since $\cE_{\infty,h}$ is the $\Gamma$-limit of $\cE_{\sigma,h}$ in the $\cT_\sigma$-weak topology by \Cref{cor:gammalimit:hto0:weak}, it suffices to show the following:
		For any sequence $\{ v_h  \}_h \subset \cT_\sigma$ with $\sup_{h > 0} \cE_{\sigma,h}(v_h) < \infty$, there exists $v \in \cT_\sigma$ such that $v_h \rightharpoonup v$ weakly in $\cT_\sigma$. 
		%$\lim\limits_{h \to 0} \vnorm{v_h - v}_{\cT_\si} = 0$.
		But this is apparent from the bound on $\sup_{h > 0} \vnorm{ v_h }_{\cT_\si} $ implied by \Cref{assumption2}(ii).
		Thus the result follows by applying the $\Gamma$-convergence framework of \cite[Theorem 1.21]{braides2002gamma}.

		%Finally, 
        Both \eqref{eq:gammaconvergence3} and \eqref{eq:gammaconvergence4} are proved similarly, by applying the framework of \cite[Theorem 1.21]{braides2002gamma}. We use \Cref{theorem:gammaconvergence:sigmatoinf} (resp. \Cref{theorem:gammaconvergence:ac}) and the relative compactness of any sequence $\{ v_\sigma \in W_{\sigma,h} \}$ with $\sup_{\sigma \in (0,\infty)} \cE_\sigma(v_\sigma) < \infty$ (resp. $\{ v_n \in W_{\sigma_n,h_n} \}$ with $\sup_{n \in \bbN} \cE_{\sigma_n}(v_n) < \infty$) implied by \Cref{assumption1}(ii) and \Cref{assumption2}(ii).
	\end{proof}

Even if the energies $\cE_\si$ do not satisfy \Cref{rmk:weakimpliesstr}, the mode of convergence in item 1) can still be improved if the first variations satisfy the following assumption:

\begin{assumption}\label{assumption5}
There exists a function $Q_p \in C^0([0,\infty)^2)$ independent of $\sigma$ such that,  the first variation $\cF_\sigma(u) : \cT_\sigma \to \cT_{-\sigma}$ 
at $u \in \cT_\si$,
defined via $\vint{ \cF_\sigma(u), v} := \frac{d}{dt} \cE_\sigma(u+t w) \big|_{t=0}$ for $w \in \cT_\si$ satisfies: for the solution $u_\si \in  \cT_\si$  of \eqref{eq:minprob} and 
any $v\in  \cT_\si$,
\begin{equation}\label{eq:FirstVariationEstimates:1}
				\begin{split}
					&\vnorm{\cF_\sigma(u_\si) - \cF_{\sigma}(v)}_{\cT_{-\sigma}} \\
					&\leq Q_p( \vnorm{u_\si}_{\cT_\sigma}, \vnorm{v}_{\cT_\sigma})
					\begin{cases}
						| \vint{ \cF_\sigma(u_\si)-\cF_\sigma(v),u_\si-v}|^{\frac{p-1}{p}}, &\text{ if } 1 < p < 2, \\
						|\vint{ \cF_\sigma(u_\si)-\cF_\sigma(v),u_\si-v}|^{1/2}, &\text{ if } p \geq 2,
					\end{cases} \\
				\end{split}
			\end{equation}
		and
			\begin{equation}\label{eq:FirstVariationEstimates:2}
				\begin{split}
					&\vnorm{u_\si-v}_{\cT_\sigma} \\
					&\leq Q_p( \vnorm{u_\si}_{\cT_\sigma}, \vnorm{v}_{\cT_\sigma})
					\begin{cases}
						|\vint{ \cF_\sigma(u_\si)-\cF_\sigma(v),u_\si-v}|^{1/2}, &\text{ if } 1 < p < 2, \\
						| \vint{ \cF_\sigma(u_\si)-\cF_\sigma(v),u_\si-v} |^{1/p}, &\text{ if } p \geq 2.
					\end{cases}
				\end{split}
			\end{equation}
\end{assumption}

With this, we have the following:

\begin{theorem}[Strong convergence with a fixed $\sigma \in (0,\infty\text{]}$ as $h \to 0$]
    Take all the assumptions of the previous theorem. If additionally \Cref{assumption5} is satisfied, then $u_{\si,h} \to u_\si$ as $h \to 0$ in the $\cT_\si$-strong topology, and there exists a constant $C >0$, independent of $h$, such that
    \begin{equation}\label{eq:gammaconvergence2}
		\vnorm{u_{\si,h}-u_\si}_{\cT_\si} \leq
		C 
		\begin{cases}
		  \inf_{v_{\si, h}\in W_{\si, h}}
		  \| v_{\si, h}-u_{\si}\|_{\cT_{\si}}^{p/2}\,, &\text{ if } 1 < p < 2\,, \\
		  \inf_{v_{\si, h}\in W_{\si, h}}
		  \| v_{\si, h}-u_{\si}\|_{\cT_{\si}}^{2/p}\,, &\text{ if } p \geq 2\,.
		\end{cases}
    \end{equation}
\end{theorem}

\begin{proof}
%Next we prove \eqref{eq:gammaconvergence2}. 
		The standard properties of the Ritz-Galerkin approximation applied to the Euler-Lagrange equations gives
		\begin{equation*}
			\begin{split}
				\vint{\cF_\si(u_{\si,h}) - \cF_\si(u_\si),u_{\si,h} - u_{\si}} 
				&= -\vint{\cF_\si(u_{\si,h}) - \cF_\si(u_\si),u_{\si}} \\
				&=\vint{\cF_\si(u_{\si,h}) - \cF_\si(u_\si),v_{\si,h} - u_{\si}} \quad \forall v_{\si,h} \in W_{\si,h}.
			\end{split}
		\end{equation*}
		Therefore,
		\begin{equation*}
			\begin{split}
				&|\vint{\cF_\si(u_{\si,h}) - \cF_\si(u_\si),u_{\si,h} - u_{\si}}| \\
				&= |\vint{\cF_\si(u_{\si,h}) - \cF_\si(u_\si),v_{\si,h} - u_{\si}}| \\
				&\leq \Vnorm{ \cF_\si(u_{\si,h}) - \cF_\si(u_\si) }_{\cT_{-\si}} \vnorm{ v_{\si,h} - u_{\si} }_{\cT_\si}\,.
			\end{split}
		\end{equation*}
		Applying \eqref{eq:FirstVariationEstimates:1} in  \Cref{assumption5} and absorbing the resulting term, we get
		\begin{equation*}
			\begin{split}
				|\vint{\cF_\si(u_{\si,h}) - \cF_\si(u_\si),u_{\si,h} - u_{\si}}|^{\frac{1}{p}}
				&\leq Q_p( \vnorm{u_\si}_{\cT_\si}, \vnorm{u_{\si,h}}_{\cT_\si} ) \vnorm{u_\si - v_{\si,h}}_{\cT_\si}
        \end{split}
        \end{equation*}
        for $1 < p < 2$ while for $p \geq 2$
        \begin{equation*}
            \begin{split}
				|\vint{\cF_\si(u_{\si,h}) - \cF_\si(u_\si),u_{\si,h} - u_{\si}}|^{\frac{1}{2}} 
				&\leq Q_p( \vnorm{u_\si}_{\cT_\si}, \vnorm{u_{\si,h}}_{\cT_\si} ) \vnorm{u_\si - v_{\si,h}}_{\cT_\si}.
			\end{split}
		\end{equation*}
		Then using the uniform bounds on the $\cT_\si$ norms of $u_\si$ and $u_{\si,h}$, we place this estimate into \eqref{eq:FirstVariationEstimates:2} to get
		\begin{equation*}
			\begin{split}
				\vnorm{u_{\si} - u_{\si,h}}_{\cT_\si}
				&\leq \vnorm{u_\si - v_{\si,h}}_{\cT_\si}^{p/2}, \qquad 1 < p < 2, \\
				\vnorm{u_{\si} - u_{\si,h}}_{\cT_\si}
				&\leq \vnorm{u_\si - v_{\si,h}}_{\cT_\si}^{2/p}, \qquad p \geq 2. \\
			\end{split}
		\end{equation*}    
\end{proof}

%  \begin{remark}\label{rmk:bestapproximation}
%We can interpret the error estimates given in the above theorems as some forms of the best approximation properties, for the case $p=2$. One may also get similar conclusions for $p\neq 2$ under suitable modifications of \Cref{assumption5}. \end{remark}

\section{Application to nonlocal models with heterogeneous boundary localization and their FEM approximations}\label{sec:example}
	
	For $d\geq 1$, we let $\Omega \subset \R^d$ be an open connected set (a domain) that is bounded and Lipschitz. 
	On the domain $\Omega$, we consider a family of nonlocal function spaces involving boundary localization as introduced in 
	\cite{scott2023nonlocal}.
	
	Before any technical discussion, let us make a general comment. As noted earlier, a special feature of the framework of AC schemes developed in \cite{tian2014sinum,tian2020sirv} is to utilize the increased regularity of solution in the limiting process. Intuitively, the boundary localization enhances solution regularity, thus, we expect to be able to extend the framework to the case here. However, the boundary localization brings out additional complications such as the singularity of the nonlocal interaction kernel at the boundary. Thus, we need to carry out careful estimates to ensure the conclusions remain valid. This is achieved, largely with the help of the analytical studies in \cite{scott2023nonlocal,scott2023nonlocal2}.

	\subsection{Nonlocal function spaces}\label{sec:FunctionSpaces}
	We follow \cite{scott2023nonlocal} to present more general spaces associated with $L^p$ spaces for some
	given exponent $p \in (1,\infty)$  and $\beta \in [0,d+p)$ and constant $\delta$, which could be used to study nonlinear problems.
	First, we introduce the Banach space
	\begin{equation*}
		\mathfrak{W}^{\beta,p}[\delta;q](\Omega) := \{ u \in L^p(\Omega) \, :\, [u]_{ \mathfrak{W}^{\beta,p}[\delta;q](\Omega) } < \infty \}\,, 
	\end{equation*}
	equipped with the norm determined by
	$
	\Vnorm{u}_{\mathfrak{W}^{\beta,p}[\delta;q](\Omega)}^p := \Vnorm{u}_{L^p(\Omega)}^p + [u]_{\mathfrak{W}^{\beta,p}[\delta;q](\Omega)}^p$, where
	the nonlocal seminorm is defined by 
	\begin{equation}\label{eq:Intro:NonlocalSeminorm}
		[u]_{\mathfrak{W}^{\beta,p}[\delta;q](\Omega)}^p =   \int_{\Omega} \int_{\Omega} \gamma_{\beta,p}[\delta;q](\bx,\by) |u(\by)-u(\bx)|^p \, \rmd \by \, \rmd \bx\,,
	\end{equation}
	The constant
	$\delta > 0$ is the \textit{bulk horizon parameter} and
	%$\delta \in (0,\frac{1}{3})$, 
	the kernel in \eqref{eq:Intro:NonlocalSeminorm} is defined as
	\begin{equation}\label{eq:Intro:kernelgamma}
		\gamma_{\beta,p}[\delta;q](\bx,\by) := \mathds{1}_{ \{ |\by-\bx| < \delta q(\dist(\bx,\p \Omega)) \} } \frac{ C_{d,\beta,p}^\gamma }{ |\bx-\by|^{\beta} } \frac{1}{ (\delta q(\dist(\bx,\p \Omega)))^{d+p-\beta} }\,.
	\end{equation}
	For a Lebesgue measurable set $A \subset \R^d$, $\mathds{1}_A$ defines its standard characteristic function.
	%The constant 
	$C_{d,\beta,p}^\gamma > 0$ is 
	%defined to normalize the ``scale-invariant'' kernel corresponding to the $p$-th moment of $\gamma_{\beta,p}[\delta;q]$.  That is, $C_{d,\beta,p}$ is chosen 
	a normalization constant
	so that 
	for any $\bx \in \Omega$, 
	\begin{equation}\label{eq:Intro:StdKernelNormalization}
		\int_{B(0,1)} \gamma_{\beta,p}[\delta;q](\bx,\by)|\bx-\by|^p \, \rmd \by = \int_{B(0,1)} \frac{ C_{d,\beta,p}^\gamma }{ |\bsxi|^{\beta-p} } \, \rmd \bsxi = \frac{ \sqrt{\pi} \Gamma ( \frac{d+p}{2} ) }{ \Gamma(\frac{p+1}{2} ) \Gamma(\frac{d}{2}) } := \overline{C}_{d,p}\,,
	\end{equation}
    %where $\overline{C}_{d,p}$ is defined in \eqref{eq:ConstantHorizon:Kernel}
    %From \Cref{eq:Intro:StdKernelNormalization}, we get $C_{d,\beta,p} = \overline{C}_{d,p} \frac{d+p-\beta}{\sigma(\bbS^{d-1})}$, where $\sigma$ denotes surface measure and $\bbS^{d-1} \subset \R^d$ is the unit sphere. 
    with $B(0,1)$ denoting the unit ball centered at the origin in $\mathbb{R}^d$ and $\Gamma(z)$ denoting the Euler gamma function.
    To control the rate of localization at the boundary, a rate function $q$ is introduced in  \Cref{eq:Intro:kernelgamma}, and will be specified later. 
    %Specific assumptions on $q$ are given later, see \Cref{assump:NonlinearLocalization} and \Cref{assump:MomentsOfNonlinLoc}.
    
    The nonlocal problem we treat is a minimization problem with an inhomogeneous Neumann-type constraint on $\p \Omega$ (although other common boundary conditions can be considered; see \cite{scott2023nonlocal,scott2023nonlocal2}, and \Cref{rmk:Dirichlet} below). 
	For this, we follow \cite{scott2023nonlocal} to define
	\begin{equation}\label{eq:HomNonlocSpDef}
		\begin{split}
			\mathring{\mathfrak{W}}^{\beta,p}[\delta;q](\Omega) := \left\{ u \in \mathfrak{W}^{\beta,p}[\delta;q](\Omega) \, : \, (u)_\Omega := \frac{1}{|\Omega|}\int_{\Omega} u(\bx) \, \rmd \bx = 0  \right\}\,.
		\end{split}
	\end{equation}
	
	To apply the abstract framework in \Cref{sec:ac}, we let 
	\begin{defn}\label{definitionspace}
		For $1<p<\infty$ and a maximum threshold $\delta_0$ defined below,
		\beq
		\label{eq:spacefamily}
		\cT_\si = \left \{
		\begin{aligned}
			& \mathring{W}^{1,p}(\Omega)\quad \text{for } \si = \infty,\\[2pt]
                & \mathring{\mathfrak{W}}^{\beta,p}[\min\{ \delta_0, 1/\si \};q](\Omega) \quad \text{for } \si \in (0,\infty),\\[2pt]
			%&\mathring{\mathfrak{W}}^{\beta,p}[1/\sigma;q](\Omega) \quad \text{for } \si \in [1/\delta_0, \infty),\\[4pt]
			%&\mathring{\mathfrak{W}}^{\beta,p}[\delta_0;q](\Omega) \quad \text{for } \si \in(0, 1/\delta_0),\\[4pt]
			&L^p(\Om) \quad \text{for } \si =0,\\[2pt]
   			&\{ f \in [\mathfrak{W}^{\beta,p}[\min \{ \delta_0, 1/\si \};q](\Omega)]^* \, : \, \vint{f,1} = 0 \} \quad \text{for } \si \in (-\infty,0), \\[2pt]
			%&\{ f \in [\mathfrak{W}^{\beta,p}[\delta_0;q](\Omega)]^* \, : \, \vint{f,1} = 0 \} \quad \text{for } \si \in (-1/\delta_0,0), \\[4pt]
			%&\{ f \in [\mathfrak{W}^{\beta,p}[1/\sigma;q](\Omega)]^* \, : \, \vint{f,1} = 0 \} \quad \text{for } \si \in (-\infty,-1/\delta_0], \\[4pt]
			&\{ f \in [W^{1,p}(\Omega)]^* \, : \, \vint{f,1} = 0 \} \quad \text{for } \si = -\infty. 
		\end{aligned}
		\right.
		\eeq
	\end{defn}
	
	This choice of dual space permits the consideration of more than just distributional derivatives. For example, functions $f_0 \in L^{\frac{p}{p-1}}(\Omega)$, $\bff_1 \in L^{\frac{p}{p-1}}(\Omega;\bbR^d)$ and $g \in L^{\frac{p}{p-1}}(\p \Omega)$ can define a functional $f \in [W^{1,p}(\Omega)]^*$ via
	\begin{equation*}
		\vint{f,u} := \int_{\Omega} f_0(\bx) u(\bx) \, \rmd \bx + \int_{\Omega} \bff_1(\bx) \cdot \grad u(\bx) \, \rmd \bx + \int_{\p \Omega} g(\bx) u(\bx) \, \rmd S(\bx),
	\end{equation*}
	and $f \in \cT_{-\infty}$ provided $$\int_{\Omega} f_0(\bx) \, \rmd \bx + \int_{\p \Omega} g(\bx) \, \rmd S(\bx) = 0.$$
	
	\subsection{Definitions and assumptions on the nonlocal models}
	Following \cite{scott2023nonlocal}, to characterize the dependence of the localization on the distance function and to define the nonlocal problems, a rate function $q: [0,\infty) \to [0,\infty)$ and a generalized distance function $\lambda : \overline{\Omega} \to [0,\infty)$ 
    are introduced,
    along with 
    %some parameters and functions, including in particular 
    a bulk horizon parameter $\delta>0$ and a kernel function $\rho$.
    %and a mollifier $\psi$,
    For the purposes of this paper, we define
    $$
    q(r) := r \left( \frac{1}{1+\rme^{-r^2}} - \frac{1}{2} \right), \qquad r \in [0,\infty).
    $$
    The generalized distance $\lambda$, the kernel $\rho$, and the horizon $\delta$ are assumed to satisfy:
    % \begin{equation}
    %     \begin{cases}
    %         i) & q \in C^k([0,\infty)) \text{ for } k = k_q \in \bbN \cup \{\infty\}, \\
    %         ii) & q(0) = 0, \quad  0 < q(r) \leq r \: \forall r > 0, \\
    %         iii) & q' \in C^1_b([0,\infty)) \text{ with } q'(0) = 0, \quad 0 \leq q'(r) \leq 1 \: \forall r \geq 0, \\
    %         &\quad \text{ and } \exists c_q > 0 \text{ such that } q'(r) > 0 \: \forall r \in (0,c_q]; \\
    %     \end{cases}
    % \end{equation}
	%
	% \begin{align} \label{assump:Localization}
	% 	\left\{
	% 	\begin{aligned} 
	% 		i) & \,\exists\,  \kappa_0 \geq 1 \text{ such that }
	% 		\frac{1}{\kappa_0} \dist(\bx,\p \Omega) \leq \lambda(\bx) \leq \kappa_0 \dist(\bx,\p \Omega),\; \forall  \bx \in \overline{\Omega}\,;\\
	% 		%ii) & \,\exists\,  \kappa_1 > 0 \Rightarrow  |\lambda(\bx) - \lambda(\by)| \leq \kappa_1 |\bx-\by|, \; \forall \bx,\by \in \Omega; \\
	% 		ii) & \, \lambda \in C^0(\overline{\Omega}) \cap C^{\infty}(\Omega) 
 %            %\text{ for } k=k_\lambda \in \bbN_0 \cup \{\infty\}; 
 %            \text{ and } \\
	% 		%iii) & \, \text{if } k = 0, \text{ then } \exists \kappa_1 > 0 \text{ such that } |\lambda(\bx) - \lambda(\by)| \leq \kappa_1 |\bx-\by| \; \forall \bx,\by \in \Omega\,, \\
	% 		iv) & \, \forall \alpha \in \bbN^d_0, 
 %                %|\alpha| \leq k\,, 
	% 		\exists\,  \kappa_{\alpha} > 0 \Rightarrow |D^\alpha \lambda(\bx)| \leq \kappa_{\alpha} |\dist(\bx,\p \Omega)|^{1-|\alpha|}, \;  \forall \bx \in \Omega\,.
	% 	\end{aligned}\tag{\ensuremath{\rmA_{\lambda}}}\right.	%
	% \end{align}
\begin{align}\label{assump:Localization}
		\left\{
		\begin{aligned} 
			i) & \,\exists\,  \kappa_0 \geq 1 \text{ such that }
			\frac{1}{\kappa_0} \dist(\bx,\p \Omega) \leq \lambda(\bx) \leq \kappa_0 \dist(\bx,\p \Omega),\; \forall  \bx \in \overline{\Omega}\,;\\
			ii) & \, \lambda \in C^0(\overline{\Omega}) \cap C^{\infty}(\Omega) \text{ and } \forall \alpha \in \bbN^d_0, \; \exists\, \kappa_{\alpha} > 0 \text{ such that } \\
            &|D^\alpha \lambda(\bx)| \leq \kappa_{\alpha} |\dist(\bx,\p \Omega)|^{1-|\alpha|} \;  \forall \bx \in \Omega\,.
            \end{aligned}
            \tag{\ensuremath{\rmA_{\lambda}}}\right.	
	\end{align}
	\begin{align}\label{assump:VarProb:Kernel}
		\left\{  
		\begin{gathered}
			\rho \in L^{\infty}(\R)\,, \;     \supp \rho \Subset (-1,1)\,,\;  \frac{1}{ \overline{C}_{d,p} }\int_{B(0,1)} |\bz|^{p-\beta} \rho(|\bz|) \, \rmd \bz = 1\, ;\\
			\rho(-x) = \rho(x),\, \rho(x)\geq 0 \text{ and nonincreasing } \forall x \in  [0,1)\,;\\
			\exists\, \text{ constants } c_{\rho} \in (0,1) \text{ and } C_{\rho}>0\; \text{such that }\, \rho(x)\geq C_\rho, \; \forall x\in  [-c_{\rho}, c_{\rho}].
		\end{gathered}
		\tag{\ensuremath{\rmA_{\rho}}}
		\right.
	\end{align}
	%\begin{align}\label{eq:assump:Phi}
	%	\left\{       	\begin{gathered}
		%	\Phi \text{ is convex,\ and }\, \exists \text{ positive constants $c$ and $C$, such that}\\
		%	  \forall \, t \in \bbR\,, \; \Phi_p(t) = \Phi_p(-t)\,,\; c|t|^p \leq \Phi_p(t) \leq C |t|^p  \,, \\  
		%    |\Phi_p'(t)| \leq C |t|^{p-1}\,, \;
		%    \Phi_p'(t)t \geq |t|^p\,, \, \text{ and } \, c |t|^{p-2} \leq |\Phi_p''(t)| \leq C|t|^{p-2}\,.
		%	    \end{gathered}
	% \right.
	%  \tag{\ensuremath{A_\Phi}}
	%\end{align}
	\begin{align}\label{eq:HorizonThreshold2}
		%\left\{     
		%\begin{gathered} 
		  \delta \in (0, \delta_0) \; \text{ with } \; \delta_0 
            %:= \min\{\underline{\delta}_0,\bar{\delta}_0\}\,, \text{ where }\;
		  %\underline{\delta}_0 := \frac{1}{3 \max \{ 1, \kappa_1, C_q \kappa_0^{\log_2(C_q)} \} }\;\\ 
			%\text{ and $\bar{\delta}_0$ is the smallest positive root of } 
		  %M_q(\delta) = \frac{1}{3} \text{ for } M_q(\delta) := \frac{(1+\kappa_1 \delta) \delta}{(1-\kappa_1 \delta)^2}.
		\tag{\ensuremath{\rmA_{\delta}}}
		%\end{gathered}
        := \frac{1}{3 \max \{ 1, \kappa_1, 8 \kappa_0^{3} \} }\;
		%\right.
	\end{align}
        More general $q$ and $\lambda$ can be considered. In fact, the $q$ defined above is an archetype of a more general class of rate functions; for instance see \cite{scott2023nonlocal}
        %specified by the more general conditions 
        %\Cref{assump:NonlinearLocalization} and \Cref{assump:MomentsOfNonlinLoc}; see the appendix 
        for precise details and further discussion.
        
        With all of the parameters now specified, we 
        introduce 
        %the nonlocal energies We consider 
        variational problems associated to the nonlocal energies
	\begin{equation}\label{eq:fxnal:delta}
		G_{p,\delta}(u) := \int_{\Omega} \int_{\Omega} \frac{ \rho ( |\by-\bx|/\eta_\delta(\bx))}{ \eta_\delta(\bx)^{d+p-\beta}
			|\bx-\by|^{\beta-p}
		} \Phi_p \Big( \frac{ |u(\bx)-u(\by)|}{|\bx-\by|} \Big) \, \rmd \by \, \rmd \bx,
	\end{equation}
        where the heterogeneous localization function $\eta$ is given by
        \begin{equation}\label{eq:localizationfunction}
	\eta_\delta(\bx) := \delta \eta(\bx) := \delta  q(\lambda(\bx)), \quad \bx \in \Omega.
        \end{equation}
	For the sake of simplicity we set $\Phi_p(t) = t^p/p$, and for $\delta = \min\{ \delta_0, 1/\sigma \} > 0$ for $\sigma \in (0,\infty)$ define the functionals $\cG_\sigma$ as
	\begin{equation}\label{eq:fxnal}
		\cG_\sigma(u) := G_{p,\delta}(u) = \frac{1}{p} \int_{\Omega} \int_{\Omega} \frac{ \rho ( |\by-\bx|/\eta_\delta(\bx))}{ \eta_\delta(\bx)^{d+p-\beta}
			|\bx-\by|^{\beta}
		}  |u(\bx)-u(\by)|^p \, \rmd \by \, \rmd \bx\,.
	\end{equation}
	We finally define the parametrized energies for $\si \in (0,\infty)$ as
	\begin{equation}\label{eq:energy}
		\cE_\sigma(u) := \cG_\sigma(u) - \vint{f, K_{\delta} u}, \qquad \delta = \min\{\delta_0,1/\si\}, \; u \in \cT_\si, \text{ and } f \in \cT_{-\infty},
	\end{equation}
    where $K_{\delta}$ is a smoothing operator to be defined below, and we consider the variational problem \eqref{eq:minprob}.
	
	Given the above assumptions, as $\delta \to 0$, i.e., $\si\to \infty$, the functional $\cG_\sigma$ formally converges to its local counterpart, i.e.
	\begin{equation}\label{eq:LocalizedObjectsDef}
		%    \begin{equation*}
			\begin{gathered}
				\cG_\sigma(u) \to \cG_\infty(u) := \frac{1}{p} \int_{\Omega} |\grad u(\bx)|^p \, \rmd \bx\,.
				%            \quad \text{ and } \quad 
			\end{gathered}
		\end{equation}		
		Hence, the nonlocal minimization problem \eqref{eq:minprob} for $\si = \infty$ in this case is a minimization problem associated to the $W^{1,p}$-seminorm with Neumann constraint
		% \begin{equation}
		% 	\cE_\infty(u) = \min_{ v \in \cT_\infty } \cE_\infty(v), 
		% \end{equation}
		%where
		\begin{equation}\label{eq:energy:local}
			\begin{split}
				\cE_\infty(v) = \cG_{\infty}(v) - \vint{f,v} %= 
				%\min_{ v \in \cT_\infty } \cG_\infty(v) - \vint{f,u} = \min_{ v \in \mathring{W}^{1,p}(\Omega) } 
				=\frac{1}{p} \int_\Omega |\grad v(\bx)|^p \, \rmd \bx - \vint{f,v}.
			\end{split}
		\end{equation}

        \begin{remark}\label{rmk:Dirichlet}
            The Neumann problem is for illustrative purposes; we can treat other types of nonlocal boundary-value problems with heterogeneous localization %problem 
            associated to a variational form. For instance, let the homogeneous space $\mathfrak{W}^{\beta,p}_0[\delta;q](\Omega)$ be the closure of $C^{\infty}_c(\Omega)$ in the $\mathfrak{W}^{\beta,p}[\delta;q](\Omega)$-norm. Then, minimizing $\cE_\si$ defined in \eqref{eq:energy} over the space $\cT_\si = \mathfrak{W}^{\beta,p}_0[\min\{\delta_0,1/\si\};q](\Omega)$ for $\si \in (0,\infty)$ with $f \in \cT_{-\infty} = W^{-1,p'}(\Omega)$ corresponds to a Dirichlet-type boundary-value problem.
            %constraint in the minimization problem \eqref{eq:minprob}. 
        \end{remark}

		\subsection{Properties of the nonlocal spaces}
		Given the above assumptions, we have
		
		\begin{proposition}[\cite{scott2023nonlocal}]
			%[Special case of \cite{scott2023nonlocal} with $p=2$]
			\label{prop:nonlocalspace}
			Under the above assumptions, we have the following.
                \begin{itemize}
			\item[\rm (1)] Embedding:
			\begin{equation}\label{eq:W1pembed}
                \displayindent0pt
			\displaywidth\textwidth
				[u]_{\mathfrak{W}^{\beta,p}[\delta;q](\Omega)} \leq \frac{1}{(1-\delta)^{1/p}} [u]_{W^{1,p}(\Omega)}\,,\quad \forall  u \in W^{1,p}(\Omega).
			\end{equation}
			\item[\rm (2)] Density of smooth functions:
                $C^{\infty}(\overline{\Omega})$ is dense in $\mathfrak{W}^{\beta,p}[\delta;q](\Omega)$.
			%$C^{k}(\overline{\Omega})$ is dense in $\mathfrak{W}^{\beta,p}[\delta;q](\Omega)$ for any $k \leq k_q$. 
			\item[\rm (3)] Trace theorem: the trace operator $T : \mathfrak{W}^{\beta,p}[\delta;q](\Omega) \to W^{1-1/p,p}(\p \Omega)$ is well defined with a constant $C = C(d,\beta,q,\Omega)$ such that,
			$$
                \displayindent0pt
			\displaywidth\textwidth
			\Vnorm{Tu}_{W^{1-1/p,p}(\p \Omega)} \leq C \Vnorm{u}_{\mathfrak{W}^{\beta,p}[\delta;q](\Omega)} ,\qquad\forall u \in \mathfrak{W}^{\beta,p}[\delta;q](\Omega)\,.
			$$ 
                Consequently,  
			$$u\in \mathfrak{W}^{\beta,p}_{0}[\delta;q](\Omega)=\{ u \in \mathfrak{W}^{\beta,p}[\delta;q](\Omega) \text{ and } T u = 0 \text{ on } \p \Omega\}.$$
			%and by the embedding \Cref{eq:W1pembed} and density of $C^{1}_c(\Omega)$ in $H^{1}_0(\Omega)$, we have
			%$$u\in \mathfrak{W}^{\beta,2}_{0}[\delta;q](\Omega)=\{ \text{closure of } C^{1}_c(\Omega) \text{ with respect to } \Vnorm{\cdot}_{\mathfrak{W}^{\beta,2}[\delta;q](\Omega)} \}\,.$$
			\item[\rm (4)] Compactness:
			let $\delta\to 0$, and let $\{ u_\delta \}_\delta \subset \mathfrak{W}^{\beta,p}[\delta;q](\Omega)$ be a sequence such that $\sup_{\delta > 0} \Vnorm{u_\delta}_{L^p(\Omega)} \leq C < \infty$, and  
			$
			\sup_{\delta > 0} [u_\delta]_{\mathfrak{W}^{\beta,p}[\delta;q](\Omega)} := B < \infty\,.
			$
			Then $\{ u_\delta \}_\delta$ is precompact in the strong topology of $L^p(\Omega)$. Moreover, any limit point $u$ belongs to $W^{1,p}(\Omega)$ with $\Vnorm{u}_{L^{p}(\Omega)} \leq C$ and $\Vnorm{\grad u}_{L^{p}(\Omega)} \leq B$.
			\item[\rm (5)] Poincare inequalities:
			% Let $1 \leq p<\infty$, $\beta \in [0,d+p)$, and $q$ satisfy \eqref{assump:NonlinearLocalization}. Then t
			there exist positive constants $C_D=C_D(d,\beta,p,q,\Omega)$ and $C_N=C_N(d,\beta,p,q,\Omega)$ such that 
			\begin{equation}\label{eq:PoincareDirichlet}
                \displayindent0pt
			\displaywidth\textwidth
				\Vnorm{u}_{L^p(\Omega)} \leq C_D [u]_{\mathfrak{W}^{\beta,p}[\delta;q](\Omega)}\,,
				\quad \forall u \in \mathfrak{W}^{\beta,p}_{0}[\delta;q](\Omega)\,,
			\end{equation}
			\begin{equation}\label{eq:PoincareNeumann}
                \displayindent0pt
			\displaywidth\textwidth
				\Vnorm{u}_{L^p(\Omega)} \leq C_N [u]_{\mathfrak{W}^{\beta,p}[\delta;q](\Omega)}\,,
				\quad \forall u \in \mathring{\mathfrak{W}}^{\beta,p}[\delta;q](\Omega)\,.
			\end{equation}
			\item[\rm (6)] Stability of norm in $\delta$: for constants $0 < \delta_1 \leq \delta_2 < \underline{\delta}_0$,
			\begin{equation*}
                \displayindent0pt
			\displaywidth\textwidth
				\begin{aligned}
					&	\left( \frac{1-\delta_2}{2(1+\delta_2)} \right)^{\frac{d+p-\beta}{p}}[u]_{\mathfrak{W}^{\beta,p}[\delta_2;q](\Omega)} 
					\leq [u]_{\mathfrak{W}^{\beta,p}[\delta_1;q](\Omega)}\\
					&\qquad \leq \left( \frac{\delta_2}{\delta_1} \right)^{1+(d-\beta)/p} [u]_{\mathfrak{W}^{\beta,p}[\delta_2;q](\Omega)}, \quad
					\forall 
					u \in \mathfrak{W}^{\beta,p}[\delta_2;q](\Omega)\,.
				\end{aligned}
			\end{equation*}
			\item[\rm (7)] Stability of the norm in kernel and localization function: for $\rho$ satisfying \eqref{assump:VarProb:Kernel} and $\lambda$ satisfying \eqref{assump:Localization}, 
            the quantity $(p G_{p,\delta}(u))^{1/p}$ defines a seminorm. Moreover,
            there exist positive constants $c$ and $C$ depending only on $d$, $\beta$, $p$, $\rho$, $q$, and $\kappa_0$ such that for any $u \in \mathfrak{W}^{\beta,p}[\delta;q](\Omega)$, 
		\begin{equation}\label{theorem:EnergySpaceIndepOfKernel}
                \displayindent0pt
			\displaywidth\textwidth
				c [u]_{\mathfrak{W}^{\beta,p}[\delta;q](\Omega)}^p \leq
				G_{p,\delta}(u) \leq C  [u]_{\mathfrak{W}^{\beta,p}[\delta;q](\Omega)}^p.
			\end{equation}
            %define the seminorm
   %              \begin{equation*}
			% 	[u]_{\mathfrak{V}^{\beta,p}[\delta;q;\rho,\lambda](\Omega)}^p := \int_{\Omega} \int_{\Omega} \bar{\gamma}_{\beta,p}[\delta;q;\rho,\lambda](\bx,\by) |u(\by)-u(\bx)|^p \, \rmd \by \, \rmd \bx\,.
			% \end{equation*}
			% Then there exist positive constants $c$ and $C$ depending only on $d$, $\beta$, $p$, $\rho$, $q$, and $\kappa_0$ such that for any $u \in \mathfrak{W}^{\beta,p}[\delta;q](\Omega)$,
			% \begin{equation}\label{theorem:EnergySpaceIndepOfKernel}
			% 	c [u]_{\mathfrak{W}^{\beta,p}[\delta;q](\Omega)} \leq
			% 	[u]_{\mathfrak{V}^{\beta,p}[\delta;q;\rho,\lambda](\Omega)} \leq C  [u]_{\mathfrak{W}^{\beta,p}[\delta;q](\Omega)}\,.
			% \end{equation}
        \end{itemize}
		\end{proposition}

		\subsection{Boundary-localized convolutions and their applications}
		A useful tool for nonlocal problems with heterogeneous localization is the boundary-localized  convolution  operators defined by 
		\begin{equation}\label{eq:ConvolutionOperator}
			K_{\delta}u (\bx) := \int_{\Omega} \frac{1}{(\eta_\delta(\bx))^d} \psi \left( \frac{|\by-\bx|}{ \eta_\delta(\bx) } \right) u(\by) \,\rmd \by , \; \bx \in \Omega,
		\end{equation}
		where the convolution kernel $\psi$ satisfies
		\begin{align} \label{Assump:Kernel}
			\left\{    \begin{gathered}
				\psi \in C^{k}(\R) \text{ for some } k \in \bbN_0 \cup \{\infty\}\,,
				%\; \psi(x) \geq 0\, 
    \text{ and } \, \psi(-x) = \psi(x) \geq 0,\;\forall x\in\R,\\
				[-c_\psi,c_{\psi}] \subset \supp \psi \Subset (-1,1) \text{ for fixed } c_{\psi} > 0\,, \; \text{ and }
				\int_{\R^d} \psi(|\bx|) \, \rmd \bx = 1\,.
			\end{gathered}
			\tag{\ensuremath{\rmA_{\psi}}}	
			\right.
		\end{align}
		%The adjoint operator $K_\delta^*$ is defined naturally.
		We recall some properties of $K_\delta$ 
		%and $K_\delta^\ast$ 
		shown in \cite{scott2023nonlocal}.
		
		\begin{proposition}[\cite{scott2023nonlocal,scott2023nonlocal2}]\label{theorem:convolution-estimate}
			Given the above assumptions, there exists a constant $C$ depending only on $d$, $\beta$, $p$, $\psi$, $q$, $\kappa_0$, $\kappa_1$ and $\Omega$ such that: $ \forall u \in \mathfrak{W}^{\beta,p}[\delta;q](\Omega)$, 
			\begin{align}
				&	\label{eq:KdeltaError}
				\Vnorm{u - K_{\delta} u }_{L^p(\Omega)} \leq 
				C \delta q(\diam(\Omega)) [u]_{\mathfrak{W}^{\beta,p}[\delta;q](\Omega)}\,,%\\
				%& \label{eq:Intro:KdeltaError:Frac}
				%[u - K_\delta u]_{W^{ (\beta-d)/p,p }(\Omega)} \leq C (\delta q(\diam(\Omega)) )^{ 1 - \frac{\beta-d}{p} } [u]_{ \mathfrak{W}^{\beta,p}[\delta;q](\Omega) }
				%, \quad \beta > d.
			\end{align}
			\begin{equation}\label{eq:ConvergenceOfConv}
				\lim\limits_{\delta \to 0} \Vnorm{K_{\delta} u - u}_{L^p(\Omega)} = 0\,, \quad \forall   u \in L^p(\Omega). 
			\end{equation}
			If in addition \eqref{Assump:Kernel} is satisfied for $k=k_\psi \geq 1$, then
			\begin{equation}\label{eq:Intro:ConvEst:Deriv}
				\begin{split}
					\Vnorm{ K_{\delta} u }_{W^{1,p}(\Omega)} 
					\leq C \Vnorm{u}_{\mathfrak{W}^{\beta,p}[\delta;q](\Omega)}\,,
					\qquad \forall u \in \mathfrak{W}^{\beta,p}[\delta;q](\Omega).
				\end{split}
			\end{equation} 
			\begin{equation}\label{eq:ConvergenceOfConv:W1p}
				\lim\limits_{\delta \to 0} \Vnorm{K_{\delta} u - u}_{W^{1,p}(\Omega)} = 0\,, \quad \forall     u \in W^{1,p}(\Omega). 
			\end{equation}
			Moreover,  we have $T K_{\delta} u = T u$ for all $u \in \mathfrak{W}^{\beta,p}[\delta;q](\Omega)$.
		\end{proposition}

		As an application of the operators like $K_\delta$, we can verify \Cref{assumption1} and \Cref{assumption2}.

		\begin{theorem}\label{theorem:verification1and2}
			The family of spaces and functionals defined by \Cref{eq:spacefamily}, \Cref{eq:energy} and \Cref{eq:LocalizedObjectsDef} satisfy  \Cref{assumption1} and \Cref{assumption2}. Additionally, there exists a constant $M_2 > 0$ independent of $\si$ such that
			\begin{equation}\label{eq:EmbeddingSimp}
				\vnorm{u}_{\cT_\si} \leq M_2 \vnorm{u}_{\cT_\infty} \quad \forall u \in \cT_\infty.
			\end{equation}
		\end{theorem}
		
		\begin{proof}
			\Cref{assumption1}(i) follows from the definition of the spaces, and \eqref{eq:EmbeddingSimp} follows from \Cref{eq:W1pembed}.  \Cref{assumption1}(ii) can be shown similarly to item (4) of
			\Cref{prop:nonlocalspace}. The idea is to use the boundary-localized convolution operator $K_{\delta}u$ given by
			\Cref{eq:ConvolutionOperator}, as in \cite{scott2023nonlocal}. Noticing that $K_{\epsilon}u_\delta\in W^{1,p}(\Omega)$ is uniformly bounded in $W^{1,p}(\Omega)$ with respect to $\delta$ and $\epsilon$ as $\delta\to 0$ and $\epsilon\to 0$, we obtain that the limit point $u$ belongs to $W^{1,p}(\Omega)$ with the bound
			\[ \Vnorm{u}_{L^{p}(\Omega)}\leq    \sup_{\delta > 0}
			\Vnorm{u_\delta}_{L^p(\Omega)}, \; \text{ and  }\; \Vnorm{\grad u}_{L^{p}(\Omega)} \leq
			\sup_{\delta > 0} [u_\delta]_{\mathfrak{W}^{\beta,p}[\delta;q](\Omega)} .\]
			Since $u_\delta \to u$ in $L^p(\Omega)$ and $(u_\delta)_{\Omega} = 0$ for all $\delta$, we get $(u)_\Omega = 0$.
			Finally, \Cref{assumption2} (with $\varphi(s,t) = C(1 + s^{p-1} + t^{p-1})$ for a constant $C$ independent of $\si$, and $\psi(t) = C\vnorm{f}_{[W^{1,p}(\Omega)]^*} t$), follows from \eqref{eq:PoincareNeumann}, (7) of \Cref{prop:nonlocalspace}, and \eqref{eq:Intro:ConvEst:Deriv}.
		\end{proof}

		\subsection{Gamma Convergence of nonlocal functionals to the local limit}
		
		To verify the properties of the functionals $\cE_\si$ defined by  \Cref{eq:fxnal}, we recall some results of \cite{scott2023nonlocal,scott2023nonlocal2}.

		% \begin{theorem}[\cite{scott2023nonlocal}, vanishing horizon]
			%     Let $F \in [W^{1,p}(\Omega)]^*$ with $F(1)=0$.
			%     For each $\delta > 0$, let $u_\delta \in \mathring{\mathfrak{W}}^{\beta,p}[\delta;q](\Omega)$ be the unique function that satisfies
			%     \begin{equation*}
				%         E_{p,\delta}(u_\delta) - F(K_\delta v) = \min_{ v \in \mathring{\mathfrak{W}}^{\beta,p}[\delta;q](\Omega) } E_{p,\delta}(v) - F(K_\delta v)\,.
				%     \end{equation*}
			%     Then the sequence $\{u_\delta\}_\delta$ converges strongly in $W^{\frac{(\beta-d)_+}{p},p}(\Omega)$ to the unique function $u_0 \in \mathring{W}^{1,p}(\Omega)$ that satisfies
			%     \begin{equation*}
				%         E_{p,0}(u_0) - F(u_0) = \min_{ v \in \mathring{W}^{1,p}(\Omega) } \cE_{p,0}(v) - F(v)\,.
				%     \end{equation*}
			% \end{theorem}
		
		\begin{theorem}[\cite{scott2023nonlocal}]\label{thm:GammaConvergence:Example}
			Let $f \in [W^{1,p}(\Omega)]^*$ with $\vint{f,1}=0$, and let $G_{p,\delta}$ be as in \eqref{eq:fxnal:delta}. 
			Then the functionals
			$\overline{E}_{p,\delta}$ $\Gamma$-converge in the $L^p(\Omega)$-strong topology to $\overline{E}_{p,0}$, where $\overline{E}_{p,\delta}$ and $\overline{E}_{p,0}$ 
			for $\delta \in [0,\delta_0]$ are the extensions defined by
			\begin{equation*}
				\begin{split}
					\overline{E}_{p,\delta}(u) &= 
					\begin{cases}
						G_{p,\delta}(u) - \vint{f,K_\delta u} &\text{ for } u \in \mathring{\mathfrak{W}}^{\beta,p}[\delta;q](\Omega)\,, \\
						+\infty &\text{ for } u \in L^p(\Omega) \setminus \mathring{\mathfrak{W}}^{\beta,p}[\delta;q](\Omega)\,,
					\end{cases}
					\quad \text{ and } \quad \\
					\overline{E}_{p,0}(u) &= 
					\begin{cases}
						\frac{1}{p} \int_{\Omega} |\grad u(\bx)|^p \, \rmd \bx - \vint{f,u} &\text{ for } u \in \mathring{W}^{1,p}(\Omega)\,, \\
						+\infty &\text{ for } u \in L^p(\Omega) \setminus \mathring{W}^{1,p}(\Omega)\,.
					\end{cases}
				\end{split}
			\end{equation*}
        Moreover, $\overline{E}_{p,\delta}(u) \to \overline{E}_{p,0}(u)$ for all $u \in L^p(\Omega)$.
		\end{theorem}

		\begin{corollary}\label{coro:verification3}
	       For $\sigma \in (0,\infty]$, the functionals $\cE_\sigma(u)$ defined in \eqref{eq:energy} and \eqref{eq:energy:local} satisfy \Cref{assumption3}.
		\end{corollary}
        \begin{remark}\label{rmk:GammaConvergence:Dirichlet}
            \Cref{thm:GammaConvergence:Example} and its corollary also hold for the Dirichlet problem (see \Cref{rmk:Dirichlet}), in which all the same definitions and assumptions are used except $\mathring{\mathfrak{W}}^{\beta,p}[\delta;q](\Omega)$ and $\mathring{W}^{1,p}(\Omega)$ are replaced with $\mathfrak{W}^{\beta,p}_0[\delta;q](\Omega)$ and $W^{1,p}_0(\Omega)$, respectively.
        \end{remark}

		Given the verification of \Cref{assumption1}, \Cref{assumption2}, and \Cref{assumption3}, by \cite{scott2023nonlocal}, it is immediate to get the existence and uniqueness of minimizers to \eqref{eq:minprob} 
		%for data $f_\si=K_{1/\si}^\ast f$   and $f\in  \cT_{-\infty}$.   
		with $\cE_\si$ given by \eqref{eq:energy}.
		Moreover,  minimizers $u_\si$  ($\si>0$) to the above problems converge as $\si\to \infty$, or equivalently $\delta \to 0$, in $L^p(\Omega)$, to a minimizer $u_0 \in \mathring{W}^{1,p}(\Omega)$ of the local problem. The same conclusion has been established in \cite{scott2023nonlocal} for more general nonlinear problems. 
        For completeness, we quote below the result 
        %special case associated with 
        in the special case of the example treated in this work, which has linear Poisson data.
        %problem stated for completeness. 
		
		% \begin{theorem}[\cite{scott2023nonlocal}, vanishing horizon]
			%     Let $F \in [W^{1,p}(\Omega)]^*$ with $F(1)=0$, and let $E_{p,\delta}$ be as in \eqref{eq:fxnal:delta}. 
			%     Then the functionals
			%     $\overline{E}_{p,\delta}$ $\Gamma$-converge in the $L^p(\Omega)$-strong topology to $\overline{E}_{p,0}$, where

			%     For each $\delta > 0$, let $u_\delta \in \mathring{\mathfrak{W}}^{\beta,p}[\delta;q](\Omega)$ be the unique function that satisfies
			%     \begin{equation*}
				%         E_{p,\delta}(u_\delta) - F(K_\delta v) = \min_{ v \in \mathring{\mathfrak{W}}^{\beta,p}[\delta;q](\Omega) } E_{p,\delta}(v) - F(K_\delta v)\,.
				%     \end{equation*}
			%     Then the sequence $\{u_\delta\}_\delta$ converges strongly in $W^{\frac{(\beta-d)_+}{p},p}(\Omega)$ to the unique function $u_0 \in \mathring{W}^{1,p}(\Omega)$ that satisfies
			%     \begin{equation*}
				%         E_{p,0}(u_0) - F(u_0) = \min_{ v \in \mathring{W}^{1,p}(\Omega) } \cE_{p,0}(v) - F(v)\,.
				%     \end{equation*}
			% \end{theorem}
		
		\begin{proposition}[\cite{scott2023nonlocal}]\label{theorem:WellPosedness:Neumann}
			Given $f \in \cT_{-\infty}$, i.e. $f \in [W^{1,p}(\Omega)]^*$ with $\vint{f,1} = 0$, there exists a unique minimizer $u_\si\in  \cT_\si=\mathring{\mathfrak{W}}^{\beta,p}[\delta;q](\Omega)$ of \eqref{eq:minprob} with $\cE_\si$ given by \eqref{eq:energy}
			%$G_\si$ given by \Cref{eq:fxnal}  and $f_\si=K_{1/\sigma}^\ast f$ 
			that satisfies 
			\begin{equation}\label{eq:EnergyEstimate:Neumann}
				\Vnorm{u_\si}_{\mathfrak{W}^{\beta,p}[\delta;q](\Omega)}^{p-1}
				%+ \mu \Vnorm{u_\delta}_{\mathfrak{W}^{\beta,p}[\delta;q](\Omega)}^{m-1}
				%+ \mu \int_{\Omega} |K_\delta u_\delta|^m \, \rmd \bx
				\leq C \vnorm{f}_{[W^{1,p}(\Omega)]^\ast} \,,
			\end{equation}
			for a constant $C=C(d,p,\beta,q,\lambda,\Omega,\psi)$.
			Moreover, 
                %under the further assumption that $\rho$ is nonincreasing on $[0,\infty)$,
            for a sequence $\si\to \infty$ (i.e. $\delta \to 0$), the sequence 
			$\{u_\si\}_\si$ is precompact in the strong topology on $L^p(\Omega)$. Furthermore, any limit point $u_\infty$ satisfies $u_\infty \in \cT_\infty=\mathring{W}^{1,p}(\Omega)$
			and 
			\begin{equation}\label{eq:Weak:Neumann:Local}
				\cE_{\infty}(u_\infty) = \min_{v \in \cT_\infty} \cE_\infty(v) = \min_{v \in \mathring{W}^{1,p}(\Omega)} \cE_\infty(v)\,.
				\tag{\ensuremath{\rmM_0}}
			\end{equation}
			%     That is,
			%     \begin{equation}\label{eq:BVP:NeumannLocal}
				%     \cA_\infty u_\infty = f\; \text{ in } \Omega\,,     \text{ and }    Tu_\infty= 0 \text{ on } \p \Omega\,,
				% \end{equation}
		\end{proposition}
		
        \begin{remark}
            \Cref{theorem:WellPosedness:Neumann} also holds for the Dirichlet problem associated with the nonlocal and local energy functionals (see \Cref{rmk:GammaConvergence:Dirichlet}) but without the additional compatibility constraint on $f \in W^{-1,p'}(\Omega)$.
        \end{remark}

		\subsection{Galerkin finite element discretization of the nonlocal minimization problem and the AC property}
		
		For convenience of introducing the finite element discretization, we assume that $\Omega$ has a piecewise flat boundary
		with a quasi-uniform and shape-regular triangulation $\tau_h=\{\tau\}$ parametrized by the mesh parameter $h$ which measures 
		the largest size of any element $\tau\in \tau_h$. 
		
		We set
		\[
		W_{\si, h} := \{ v \in \cT_\si : v|_\tau \in P(\tau), \quad \forall \tau \in \tau_h, (v)_{ \Omega}=0\},
		\]
		where $P(\tau)=\mathcal P_N(\tau)$ is the space of polynomials on $\tau\in\tau_h$
		of degree less than or equal to $N$ for some  $N \in \mathbb{N}$. Note that, as $\tau_h$  and $P$ are independent of $\si$, and $\cT_\infty\subset \cT_\si$ for all $\si>0$, we obviously have $Z_h<\infty$ and \Cref{assumption4}(ii) is satisfied. Still, the different regularity requirements might make $W_{\infty,h}$ a proper subset of $W_{\si,h}$, such as in the case of $\beta<d$. For the latter, it is possible to have  $W_{\si,h}$ containing discontinuous piecewise polynomials but $W_{\infty,h}$ only contain elements of $W_{\si,h}$ that are continuous over $\overline{\Omega}$.
		For more discussions on finite element approximations, we refer to \cite{brenner2008finite,ciarlet2002finite,ern2004theory}.
		The Galerkin approximation is exactly the approximation \eqref{eq:minprobnum}.
  %       to replace $\cT_\si$ by $W_{\si, h}$ in \eqref{eq:minprob}:
		% \beq  \label{eq:minprobscalarapprox}
		% \mbox{Find }\;
		% u_{\si, h}\in W_{\si, h} \text{ such that }
		% \cE_\si(u_{\si, h}) = \min_{v \in W_{\sigma,h}} \cE_\si(v).
		% \eeq

        In the case under consideration here, \Cref{assumption4'}(i) holds thanks to \eqref{eq:EmbeddingSimp}. We thus proceed to  verify \Cref{assumption4'} instead of \Cref{assumption4}, in light of \Cref{rmk:assumption4'}. 
        
        First, by definition of $W_{\si,h}$ above and by \Cref{prop:nonlocalspace}(6), we see that \Cref{assumption4'}(ii) holds. Next,
given any $v\in\cT_\si$ and	any $\epsilon>0$, 
		%we assume that $W_{\si, h}$ is dense in $\cT_\si$, i.e.,  
		since $\cT_\infty$ (or $\cT_* := C^2(\overline{\Omega})$) is dense in $\cT_\si$ with respect to the strong topology of the latter, we can find $w\in\cT_\infty$ such that
   $\|w-v\|_{\cT_\si} \leq \epsilon/2$. 
  By classical approximation theory, we can approximate any function $w\in\cT_\infty$ 
		using elements of $W_{\si,h}$ as $h\to 0$, that is,  there is some $v_h$ for a small enough $h$ satisfying $\|w-v_h\|_{\cT_\infty}\leq \epsilon/(2M_2)$, which also implies $\|w-v_h\|_{\cT_\si} \leq M_2 \|w-v_h\|_{\cT_\infty}
		\leq  \epsilon/2$. 
   Thus, we have
  $\|v-v_h\|_{\cT_\si} \leq \epsilon$. This means that
	there exists a sequence $\{v_h\in W_{\si, h}\}$ such that
		\beq \label{FEsq}
		\| v_h-v\|_{\cT_\si}\to0\quad\text{as}\quad h\to0.
		\eeq
		This gives \Cref{assumption4'}(iii). 
        \Cref{assumption4'}(iv) follows from \Cref{thm:GammaConvergence:Example}. Finally, to check \Cref{assumption4'}(v),
		for convenience, we  define a special family of spaces $\widehat W_{\si,h}$.
		
		\begin{defn} \label{defn:FEM}
			Let $\widehat W_{\infty, h} \subset W_{\infty,h}
			\subset \cT_\infty=\mathring{W}^{1,p}(\Omega)$
			be the continuous piecewise linear finite element space that corresponds to the same mesh $\tau_h$ in the definition of $W_{\si, h}$.
		\end{defn}
		
		As in \cite{tian2020sirv}, it is a fact that the continuous piecewise linear subspace of $\mathring{W}^{1,p}(\Omega)$ approximates the whole space
		%(thus $\cS_0$ also)
		as the mesh size goes to zero.  Then by \eqref{eq:W1pembed} and the density of smooth functions given in 2) of \Cref{prop:nonlocalspace}, we obtain 
        that the family
			$\widehat W_{\infty, h}$ is dense in $\cT_\si$ for any $\si\in (0, \infty]$, and so is 
			$W_{\si, h}$ if $\widehat W_{\infty, h}\subset W_{\si, h}$.
        Then since the continuous piecewise linear subspace of $\mathring{W}^{1,p}(\Omega)$ approximates the whole space as the mesh size goes to zero, by \eqref{eq:W1pembed} there exists a sequence $\{v_n \in W_{\si_n,h_n}\}$ such that $\vnorm{v_n - v}_{\cT_\infty} \to 0$ as $n \to \infty$, i.e. \Cref{assumption4'}(v) is satisfied.

		% \begin{lem}\label{lma:Assumption4satisfied}
		% 	The family
		% 	$\widehat W_{\infty, h}$ is  dense in $\cT_\si$ for any $\si\in (0, \infty]$, and so is 
		% 	$W_{\si, h}$ if $\widehat W_{\infty, h}\subset W_{\si, h}$.
		% 	That is, $W_{\si, h}$ satisfies
		% 	\Cref{assumption4}\textup{(i)}.
		% 	Moreover, the spaces $\widehat W_{\infty, h}$, and thus the spaces $W_{\si, h}$ if $\widehat W_{\infty, h}\subset W_{\si, h}$,
		% 	are asymptotically optimally approximating.
		% 	That is, they satisfy \Cref{assumption4}\textup{(ii)}.
		% \end{lem}
		
		% \begin{proof}
  %           First, $\cT_\infty$ is continuously embedded in $\cT_\si$ by \eqref{eq:EmbeddingSimp}.
  %           Second, it is shown in \cite{scott2023nonlocal} that $\cE_\infty(v) = \lim\limits_{\si \to \infty} \cE_\si(v)$ for all $v \in \cT_\infty$, i.e. the constant sequence is a recovery sequence. 
		%   Third, since $\widehat{W}_{\infty,h}$ is dense in $\cT_\si$ for all $\si$, and since the continuous piecewise linear subspace of $\mathring{W}^{1,p}(\Omega)$ approximates the whole space as the mesh size goes to zero, by \eqref{eq:W1pembed} there exists a sequence $\{v_n \in W_{\si_n,h_n}\}$ such that $\vnorm{v_n - v}_{\cT_\infty} \to 0$ as $n \to \infty$.
		% \end{proof}
		
		With all conditions of
		\Cref{assumption4'} verified, 
and the verification of
\Cref{assumption1}, \Cref{assumption2}, 
		\Cref{assumption3} already  given in 
   \Cref{theorem:verification1and2} and \Cref{coro:verification3},
		we get the  following theorem on asymptotically compatible schemes for nonlocal minimization problems involving boundary-localization.

		\begin{theorem} \label{theorem:conv1}
			Let $u_{\si, h} \in W_{\si,h}$ be a minimizer of \eqref{eq:minprobnum}, with $\cE_\si$ given by \eqref{eq:energy}, $\cT_\si$ given by \Cref{definitionspace}, and
			$\widehat W_{\infty, h}$ defined in \Cref{defn:FEM}.
			If $\widehat W_{\infty,h}\subset W_{\si, h}$, then $\| u_{\si, h}-u_\infty \|_{L^p(\Om)}\to0$ as $\si
			\to \infty$ and $h\to0$, where  $u_\infty$ is the solution of \eqref{eq:Weak:Neumann:Local} with 
   $\cE_\infty$ given by \eqref{eq:energy:local}.
		\end{theorem}
		\begin{theorem}\label{theorem:conv1an}
			Take all the assumptions of the previous theorem. Let $u_{\si, h}$ and $u_{\infty,h}$ be discrete minimizers satisfying \eqref{eq:minprobnum}.
            %with $\si<\infty$ and $\si=\infty$, respectively, with $\cE_\infty$ given by \eqref{eq:energy:local} and $\widehat W_{\infty, h}$ defined in \Cref{defn:FEM}.
			Then, for fixed $h$,
			we have $\|u_{\si, h}-u_{\infty,h} \|_{L^p(\Omega)} \to 0$ as $\si \to \infty$.
		\end{theorem}
		
		The energies $\cE_\si$ satisfy \eqref{eq:FirstVariationEstimates:1}-\eqref{eq:FirstVariationEstimates:2} with $Q_p(s,t) = (s+t)^{\frac{p-2}{2}}$ in the case $p \geq 2$, and $Q_p(s,t) = (s+t)^{\frac{2-p}{2}}$ in the case $1 < p < 2$. This is because the function $\Phi_p(t) = |t|^{p-2} t$ for $t \in \bbR$ satisfies 
		%$(\Phi_p(s) - \Phi_p(t))(s-t) \geq 0$ for all $s$, $t \in \bbR$, and satisfies
		\begin{equation*}
			\begin{split}
				|\Phi_p(s) - \Phi_p(t)|
				\leq C
				\begin{cases}
					|(\Phi_p(s) - \Phi_p(t))(s-t)|^{(p-1)/p}, & \quad \text{ for } 1 < p < 2, \\
					|(\Phi_p(s) - \Phi_p(t))(s-t)|^{1/2} (|s| + |t| )^{\frac{p-2}{2}}, & \quad \text{ for } 2 \leq p < \infty,
				\end{cases}
			\end{split}
		\end{equation*}
		and
		\begin{equation*}
			\begin{split}
				|s-t|
				\leq C
				\begin{cases}
					|(\Phi_p(s) - \Phi_p(t))(s-t)|^{1/2} (|s|^{2-p} + |t|^{2-p} )^{1/2}, & \quad \text{ for } 1 < p < 2, \\
					|(\Phi_p(s) - \Phi_p(t))(s-t)|^{1/p}, & \quad \text{ for } 2 \leq p < \infty.
				\end{cases}
			\end{split}
		\end{equation*}
		Therefore we have the following:
	\begin{theorem}\label{theorem:conv1bn}
			Take all the assumptions of the previous theorem. Let $u_{\si, h}$ and $u_{\si}$ be minimizers satisfying \eqref{eq:minprobnum} and \eqref{eq:minprob}, respectively. 
			%Assume further that
			%$V_{\del,h}\subset \cS_\del$ is a finite element space that contains all continuous piecewise  linear functions.\ Moreover, $V_{0, h}= \cS_0 \cap ({\bigcup}_{\del \in [0,\delta_0]} V_{\del,h})$.
			Then, for fixed $\si$,
			%and $\tau_h$,
			$\|u_{\si, h}-u_{\si} \|_{\cT_{\si}} $ satisfies \eqref{eq:gammaconvergence2} and converges to $0$ as $h \to 0$.
		\end{theorem}
		
		Note that the convergence result in the previous theorem also holds by \Cref{rmk:weakimpliesstr}.

        \begin{remark}
            Following the thread of reasoning in \Cref{rmk:Dirichlet} and \Cref{rmk:GammaConvergence:Dirichlet}, it is straightforward to check that all of the corresponding versions of these convergence theorems hold for the Dirichlet problem.
        \end{remark}

\section{Application to nonlocal models on a parameter-dependent domain and their FEM approximations}
\label{sec:varying}

The next example considered is a nonlocal convex functional with superquadratic growth and Dirichlet volume constraint.
The results here generalize studies made in \cite{tian2014asymptotically,tian2020sirv} for the special case of a quadratic nonlocal energy functionals corresponding to linear nonlocal problems with a constant horizon parameter.

\subsection{Nonlocal variational problems and nonlocal spaces}
We let $p \geq 2$, and let $\Omega \subset \bbR^d$ be a bounded domain, and we define the boundary layer $\Omega_{I_\delta} := \{ \bx \in \bbR^d \setminus \Omega \, : \, \dist(\bx,\p \Omega) \leq \delta \}$. Defining $\Omega_\delta := \Omega \cup \Omega_{I_\delta}$, we define the functional
\begin{equation*}
    \frac{1}{p} [u]_{\cS^p[\delta](\Omega)}^p := \frac{1}{p} \int_{\Omega_\delta} \int_{\Omega_\delta} \rho_\delta(|\bx-\by|)  \left| \frac{u(\bx)-u(\by)}{\delta} \right|^p \, \rmd \bx \, \rmd \by\,.
\end{equation*}
The function $\rho_\delta(|\bz|) := \delta^{-d} \rho(|\bz|/\delta)$ is a rescaling of a kernel $\rho \in L^1_{loc}((0,1))$ which satisfies \eqref{assump:VarProb:Kernel}.
%which is assumed to be nonnegative and nonincreasing, and satisfies
%\begin{equation}\label{eq:ConstantHorizon:Kernel}
%    \begin{gathered}
%    \supp \rho \Subset [0,1)\,, \quad \exists \text{ a constant } c_\rho \in (0,1) \text{ such that } [0,c_\rho) \subset \supp \rho\,, \\
%    \text{ and } \int_{B(0,1)} |\bz|^p \rho(|\bz|) \, \rmd \bz = \frac{ \sqrt{\pi} \Gamma ( \frac{d+p}{2} ) }{ \Gamma(\frac{p+1}{2} ) \Gamma(\frac{d}{2}) } := \overline{C}_{d,p}\,,
%    \end{gathered}
%\end{equation}
%with $B(0,1)$ denoting the unit ball centered at the origin in $\mathbb{R}^d$ and $\Gamma(z)$ denoting the Euler gamma function.

We define $L^p_0(\Omega_{\delta})$ to be the subspace of functions in $L^p(\Omega_\delta)$ that vanish on $\Omega_{I_\delta}$.
The energy space is the subspace of the nonlocal function space
\begin{equation}
    \cS^p[\delta](\Omega) := \{ u \in  L^p(\Omega_\delta) \, : \, [u]_{\cS^p[\delta](\Omega)} < \infty \}\,,
\end{equation}
defined as
\begin{equation}
    \cS^p_0[\delta](\Omega) := \{ u \in  \cS^p[\delta](\Omega) \, : \,  u |_{\Omega_{I_\delta}} = 0\} = L^p_0(\Omega_\delta) \cap \cS^p[\delta](\Omega)\,,
\end{equation}
with norm $\vnorm{ u }_{\cS^p[\delta](\Omega_\delta)}^p := \vnorm{u}_{L^p(\Omega)}^p + [u]_{\cS^p[\delta](\Omega)}^p$. It can be shown that $\cS^p[\delta](\Omega)$ is a separable uniformly convex Banach space, with $\cS^p[\delta](\Omega) \subset L^p(\Omega_\delta)$, and $\cS^p_0[\delta](\Omega)$ is the completion of $L^p_0(\Omega_\delta) \cap C^\infty_c(\Omega)$ with respect to the norm $\vnorm{\cdot}_{\cS^p_0[\delta](\Omega)}$ \cite{mengesha2015VariationalLimit}.

We collect the relevant properties of the nonlocal function spaces in the following proposition.
\begin{proposition}\label{prop:ConstantHorizon:FxnSpProp}
	The following hold:
	\begin{enumerate}
		% \item[1)] Embedding: 
		% \begin{equation}
			%     G_{p,\delta}(u) \leq C(d,p) [u]_{W^{1,p}(\Omega)}^p\,, \quad \forall u \in W^{1,p}(\Omega)\,.
			% \end{equation}
		\item[\rm (1)] Density of smooth functions: $C^\infty(\overline{\Omega_\delta})$ is dense in $\cS^p[\delta](\Omega)$.
		\item[\rm (2)] Compactness: Let $\delta \to 0$, and let $\{ u_\delta \in \cS^p_0[\delta](\Omega) \}_\delta$ be a sequence such that $\sup_{\delta > 0} \vnorm{u_\delta}_{L^p(\Omega)} \leq C < \infty$ and $\sup_{\delta > 0} [u_\delta]_{\cS^p[\delta](\Omega)} := B < \infty$.	
		Then $\{ u_\delta \}_\delta$ is precompact in the strong topology of $L^p(\Omega)$. Moreover, any limit point $u$ belongs to $W^{1,p}_0(\Omega)$ with $\Vnorm{u}_{L^{p}(\Omega)} \leq C$ and $\Vnorm{\grad u}_{L^{p}(\Omega)} \leq B$.
		\item[\rm (3)] Poincar\'e inequality: There exist constants $\delta_0 \in (0,1)$ and $C > 0$ such that for any $\delta < \delta_0$,
		\begin{equation*}
			\vnorm{u}_{L^p(\Omega)} \leq C [u]_{\cS^p[\delta](\Omega)}, \quad \forall u \in \cS^p_0[\delta](\Omega)\,.
		\end{equation*}
	\end{enumerate}
\end{proposition}

\begin{proof}
	(1) follows from an argument similar to that of \cite[Theorem 3.3]{Mengesha2023Linearization}.
	To see (2), note that the uniform bound on the seminorm implies that
	\begin{equation*}
		\int_\Omega \int_\Omega \rho_\delta(|\bx-\by|)  \left| \frac{u_\delta(\bx)-u_\delta(\by)}{\delta} \right|^p \, \rmd \bx \, \rmd \by \leq B\,.
	\end{equation*}
	Then by \cite[Theorem 4]{BBM} (see also \cite{mengesha2015VariationalLimit}) the result follows. Last, (3) follows from (2) and a contradiction argument; see for instance \cite{MengeshaDuElasticity} in which a similar nonlocal Poincar\'e-Korn inequality is established.
\end{proof}

%\subsection{AC framework}

    To apply the abstract framework in \Cref{sec:ac}, we set the following notation for the function spaces. 
	\begin{defn}\label{definitionspace:ConstantHorizon}
		For $2 \leq p < \infty$ and $0 < \delta \leq 1$,
		\beq
		\label{eq:spacefamily:ConstantHorizon}
		\cT_\si = \left \{
		\begin{aligned}
			& W^{1,p}_0(\Omega)\quad \text{for } \si = \infty,\\[4pt]
			&\cS^p_0[ \min \{ \delta_0, 1/\sigma \} ](\Omega) \quad \text{for } \si \in (0, \infty),\\[4pt]
		      %&\cS^{p}_0[1](\Omega) \quad \text{for } \si \in(0,1),\\[4pt]
			&L^p_0(\Om_{\delta_0}) \quad \text{for } \si = 0,\\[4pt]
			%&\{ f \in [\mathfrak{W}^{\beta,p}[\delta_0;q](\Omega)]^* \, : \, \vint{f,1} = 0 \} \quad \text{for } \si \in (-1/\delta_0,0), \\[4pt]
			%&L^{p'}(\Omega_{ \min\{1,1/\si\} }) \quad \text{for } \si \in (-\infty,0), 
            %\\%[4pt]
			%&W^{-1,p'}(\Omega) = [W^{1,p}_0(\Omega)]^* \quad \text{for } \si = -\infty. 
			&L^{p'}(\Omega_{ \min\{\delta_0,-1/\si\} }) \quad \text{for } \si \in (-\infty,0), 
			\\%[4pt]
			&L^{p'}(\Omega) \quad \text{for } \si = -\infty. 
		\end{aligned}
		\right.
		\eeq
	\end{defn}
        For $\si = 1/\delta > 0$ we define the functionals $\cG_\si$ as 
        \begin{equation}\label{eq:ConstantHorizon:Fxnal}
            \cG_\si(u) := \frac{1}{p} [u]_{\cS^p[1/\si](\Omega)}^p  - 
            %\vint{f, K_{1/\si} u}
            \vint{f, u}\,, \quad u \in \cT_\si \text{ and } f \in \cT_{-\si}\,.
        \end{equation}
%        The convolutions $K_\delta u$ for $u \in L^1_{loc}(\bbR^d)$ are defined as
%        \begin{equation}
%            K_\delta u := \frac{1}{\Psi_\delta(\bx)} \int_{\Omega_\delta} \psi_\delta(|\bx-\by|) u(\by) \, \rmd \by\,, \text{ where } \Psi_\delta(\bx) := \int_{\Omega_\delta} \psi_\delta(|\bx-\by|) \, \rmd \by\,,
%        \end{equation}
%        where the function $\psi$ is a $C^\infty$ standard mollifier.
Therefore, the parametrized energies are
\begin{equation}\label{eq:energy:constanthorizon}
    \cE_\si(u) := \cG_{\si}(u) - 
    %\vint{f,K_{1/\si} u}
    \vint{f, u}\,, \qquad u \in \cT_\si \text{ and } f \in \cT_{-\si}\,.
\end{equation}

Under the above normalization condition on $\rho$, we formally have $\cE_\si(u) \to \cE_\infty(u)$, where
\begin{equation*}
    \cE_\infty(u) := \cG_\infty(v) - \vint{f,v} = \frac{1}{p} \int_\Omega |\grad u|^p \, \rmd \bx - \vint{f,u}\,.
\end{equation*}
%So the local counterpart of the minimization problem is
%\begin{equation}
%    \text{ Find } u \in \cT_\infty \text{ such that } \cE_\infty(u) = \min_{ v \in \cT_\infty } \cE_\infty(v)\,.
%\end{equation}

\subsection{Properties of the function space and Gamma-convergence to the local limit}

%
%We also have the following results for the $\delta$-convolutions.
%
%\begin{proposition}
%    There exists a constant $C$ depending only on $d$, $p$, $\psi$, and $\Omega$ such that
%    \begin{equation}
%        \vnorm{u - K_\delta u}_{L^p(\Omega_\delta)} \leq C \delta [u]_{\cS^p[\delta](\Omega)}, \qquad \forall u \in \cS^p[\delta](\Omega),
%    \end{equation}
%    \begin{equation}
%        \lim\limits_{\delta \to 0} \vnorm{u-K_\delta u}_{L^p(\Omega_\delta)} = 0, \quad \forall u \in L^p(\Omega_\delta),
%    \end{equation}
%    \begin{equation}\label{eq:ConstantHorizon:ConvEst:Deriv}
%        \vnorm{K_\delta u}_{W^{1,p}(\Omega_\delta)} \leq C \vnorm{ u }_{\cS^p[\delta](\Omega)}, \qquad \forall u \in \cS^p[\delta](\Omega),
%    \end{equation}
%    \begin{equation}
%        \lim\limits_{\delta \to 0} \vnorm{u-K_\delta u}_{W^{1,p}(\Omega_\delta)} = 0, \quad \forall u \in W^{1,p}(\Omega_\delta)\,.
%    \end{equation}
%\end{proposition}

With these background results, we can verify \Cref{assumption1} and \Cref{assumption2}.

\begin{theorem}
    The family of spaces and functionals defined by \eqref{eq:spacefamily:ConstantHorizon} and \eqref{eq:energy:constanthorizon} satisfy \Cref{assumption1} and \Cref{assumption2}.
\end{theorem}

\begin{proof}
    \Cref{assumption1}(i) follows from the definition of the spaces. \Cref{assumption1}(ii) follows from item (2) of \Cref{prop:ConstantHorizon:FxnSpProp}.
    %; see for instance \cite{mengesha2015VariationalLimit}. 
    Finally, \Cref{assumption2} (with $\varphi(s,t) = C(1 + s^{p-1} + t^{p-1})$ for a constant $C$ independent of $\si$ and $\psi(t) = C\vnorm{f}_{\cT_{-\si}} t$) follows from item (3) of \Cref{prop:ConstantHorizon:FxnSpProp}.
\end{proof}

%\subsection{Gamma convergence} 
Furthermore, we have the Gamma-convergence to the local limit:
\begin{theorem}\label{thm:GammaConvergence:Example2}
	The functionals $\overline{\cE}_\si$ satisfy \Cref{assumption3}.
\end{theorem}

\begin{proof}
	The $\Gamma$-limsup and $\Gamma$-liminf inequalities can be proved via arguments similar to those of \cite[Lemma 12.1 and 12.2]{ponce2004new}. One point of departure is when choosing a recovery sequence; given $u \in W^{1,p}_0(\Omega)$ and $\si > 1/\delta_0$, define $u_\si : \Omega_\delta \to \bbR$ to be $u$ on $\Omega$ and $0$ on $\Omega_{I_{1/\si}}$. Then $u_\si \to u$ in $L^p(\Omega_{\delta_0})$ (extending both functions by $0$ to the ambient domain $\Omega_{\delta_0}$) and $\cE_\si(u_\si) \to \cE_\infty(u)$.
\end{proof}

\subsection{Galerkin finite element discretization and AC property}
Assume that $\Omega$ has a piecewise flat boundary
with a quasi-uniform and shape-regular triangulation $\tau_h=\{\tau\}$ parametrized by the mesh parameter $h$ which measures 
the largest size of any element $\tau\in \tau_h$.
We set
\[
\wt{W}_{h} := \{ v \in L^p(\Omega) : v|_\tau \in P(\tau), \quad \forall \tau \in \tau_h, v|_{\p \Omega}=0\},
\]
where $P(\tau)=\mathcal P_N(\tau)$ is the space of polynomials on $\tau\in\tau_h$
of degree less than or equal to $N$ for some  $N \in \mathbb{N}$. 
Thinking of each $v \in \wt{W}_{h}$ as defined on $\Omega_{\delta_0}$ via extension by $0$, we then define the finite-dimensional spaces
\begin{equation*}
	W_{\si,h} := \wt{W}_{h} \, \cap \cT_\si\,.
\end{equation*}

Note that, as $\tau_h$ and $P$ are independent of $\si$, and since $\cup_{\si \in (0,\infty]} W_{\si,h} = \wt{W}_h \cap \left( \cup_{\si \in (0,\infty]} \cT_{\si} \right)$, 
%$\cT_\infty \cup \cT_0 \subset \cT_\si$ for all $\si>0$, 
we have that $Z_h<\infty$ and \Cref{assumption4}(ii) is satisfied. 
%Still, the different regularity requirements might make $W_{\infty,h}$ a proper subset of $W_{\si,h}$, such as in the case of $\beta<d$. For the latter, it is possible to have  $W_{\si,h}$ containing discontinuous piecewise polynomials but $W_{\infty,h}$ only contain elements of $W_{\si,h}$ that are continuous over $\overline{\Omega}$.
%For more discussions on finite element approximations, we refer to \cite{brenner2008finite,ciarlet2002finite,ern2004theory}.
%The Galerkin approximation is exactly the approximation \eqref{eq:minprobnum}.
%
%In this setting, \Cref{assumption4'}(i) holds thanks to \eqref{eq:EmbeddingSimp}, and so we  verify \Cref{assumption4'} instead of \Cref{assumption4}, in light of \Cref{rmk:assumption4'}. First, by definition of $W_{\si,h}$ above and by \Cref{prop:nonlocalspace}(6), we see that \Cref{assumption4'}(ii) holds. 
Next,
given any $v\in\cT_\si$ and	any $\epsilon>0$, 
%we assume that $W_{\si, h}$ is dense in $\cT_\si$, i.e.,  
since $\cT_* := C^{1}_c(\Omega) \cap L^p_0(\Omega_{\delta_0})$ is dense in $\cT_\si$ with respect to the strong topology of the latter, we can find $w\in\cT_*$ such that
$\|w-v\|_{\cT_\si} \leq \epsilon/2$. 
By classical approximation theory, we can approximate any function $w\in\cT_*$ 
using elements of $W_{\si,h}$ as $h\to 0$, that is, there is some $v_h$ for a small enough $h$ satisfying $\|w-v_h\|_{\cT_\si}\leq \epsilon/2$. 
Thus, we have
$\|v-v_h\|_{\cT_\si} \leq \epsilon$. This means that
there exists a sequence $\{v_h\in W_{\si, h}\}$ such that
\beq \label{FEsq:constanthorizon}
\| v_h-v\|_{\cT_\si}\to0\quad\text{as}\quad h\to0.
\eeq
This gives \Cref{assumption4}(i). 
%\Cref{assumption4'}(iv) follows from \Cref{thm:GammaConvergence:Example}. 
Finally, to check \Cref{assumption4}(iii),
for convenience, we  define a special family of spaces.
%$\widehat W_{\si,h}$.

\begin{defn} \label{defn:FEM:constanthorizon}
	For $h > 0$, let $\widehat{W}_{\infty, h} \subset W_{\infty,h}
	\subset \cT_\infty = W^{1,p}_0(\Omega)$
	be the continuous piecewise linear finite element space that corresponds to the same mesh $\tau_h$ in the definition of $W_{\si, h}$. For $h =0$, define $\widehat{W}_{\infty,0} := \cT_\infty$.
\end{defn}
Let $v \in \widehat{W}_{\infty,h}$, and let $\{(\si_n,h_n)\}$ be a sequence as described in \Cref{assumption4}(iii). 
Define $v_{\si_n}$ to be the extension of $v$ to all of $\Omega_{1/\si_n}$ by $0$. Then clearly $\{v_{\si_n} \in \cT_{\si_n} \}$ is a recovery sequence by the argument in \Cref{thm:GammaConvergence:Example2}.

In the first case, the function $v_n = v_{\si_n}$ satisfies the asymptotic approximation property, and so \Cref{assumption4}(iii) is satisfied.

In the second case, it is a simple fact that the continuous piecewise linear subspace $\widehat{W}_{\infty,h}$ of $W^{1,p}_0(\Omega)$ approximates the whole space as the mesh size goes to zero. Therefore, defining $\{ w_n \in \widehat{W}_{\infty,h_n} \}$ to be an approximating sequence of $v$ and then defining $v_n$ by extending each $w_n$ to all of $\Omega_{1/\si_n}$ by zero so that $v_n \in W_{\si_n,h_n}$, we get that $\vnorm{ v_{n} - v_{\si_n}}_{\cT_{\si_n}} \leq C \vnorm{v_n-w_n}_{W^{1,p}(\Omega_{1/\si_n})} = C \vnorm{v-w_n}_{W^{1,p}(\Omega)} \to 0$ as $n \to \infty$. 
Therefore \Cref{assumption4}(iii) is satisfied.

With all \Cref{assumption1}, \Cref{assumption2}, 
\Cref{assumption3}, and
\Cref{assumption4} verified, 
we get the  following theorems on asymptotically compatible schemes for nonlocal minimization problems with a varying domain.

		\begin{theorem}[Convergence as $\si \to \infty$ and $h \to 0$]\label{theorem:conv1:constanthorizon}
			Let $u_{\si, h} \in W_{\si,h}$ be a minimizer of \eqref{eq:minprobnum}, with $\cE_\si$ given by \eqref{eq:energy:constanthorizon}, $\cT_\si$ given by \Cref{definitionspace:ConstantHorizon}, and
			$\widehat W_{\infty, h}$ defined in \Cref{defn:FEM:constanthorizon}.
			If $\widehat W_{\infty,h}\subset W_{\si, h}$, then $\| u_{\si, h}-u_\infty \|_{L^p(\Om)}\to0$ as $\si
			\to \infty$ and $h\to0$, where  $u_\infty$ is the solution of \eqref{eq:Weak:Neumann:Local} with 
   $\cE_\infty$ given by \eqref{eq:energy:local}.
		\end{theorem}
		\begin{theorem}[Convergence as $\si \to \infty$ with fixed $h$]\label{theorem:conv1an:constanthorizon}
			Take all the assumptions of the previous theorem. Let $u_{\si, h}$ and $u_{\infty,h}$ be discrete minimizers satisfying \eqref{eq:minprobnum}.
			Then, for fixed $h$,
			we have $\|u_{\si, h}-u_{\infty,h} \|_{L^p(\Omega)} \to 0$ as $\si \to \infty$.
		\end{theorem}

The energies $\cE_\si$ satisfy \eqref{eq:FirstVariationEstimates:1}-\eqref{eq:FirstVariationEstimates:2} with $Q_p(s,t) = (s+t)^{\frac{p-2}{2}}$ in the case $p \geq 2$, and $Q_p(s,t) = (s+t)^{\frac{2-p}{2}}$ in the case $1 < p < 2$, similarly to the functionals in \Cref{sec:example}. 
%This is because the function $\Phi_p(t) = |t|^{p-2} t$ for $t \in \bbR$ satisfies 
%\begin{equation*}
%	\begin{split}
%		|\Phi_p(s) - \Phi_p(t)|
%		\leq C
%		\begin{cases}
%			|(\Phi_p(s) - \Phi_p(t))(s-t)|^{(p-1)/p}, & \quad \text{ for } 1 < p < 2, \\
%			|(\Phi_p(s) - \Phi_p(t))(s-t)|^{1/2} (|s| + |t| )^{\frac{p-2}{2}}, & \quad \text{ for } 2 \leq p < \infty,
%		\end{cases}
%	\end{split}
%\end{equation*}
%and
%\begin{equation*}
%	\begin{split}
%		|s-t|
%		\leq C
%		\begin{cases}
%			|(\Phi_p(s) - \Phi_p(t))(s-t)|^{1/2} (|s|^{2-p} + |t|^{2-p} )^{1/2}, & \quad \text{ for } 1 < p < 2, \\
%			|(\Phi_p(s) - \Phi_p(t))(s-t)|^{1/p}, & \quad \text{ for } 2 \leq p < \infty.
%		\end{cases}
%	\end{split}
%\end{equation*}
Therefore we have the following:
\begin{theorem}[Convergence as $h \to 0$ with fixed $\si$]\label{theorem:conv1bn:constanthorizon}
	Take all the assumptions of the previous theorem. Let $u_{\si, h}$ and $u_{\si}$ be the unique minimizers satisfying \eqref{eq:minprobnum} and \eqref{eq:minprob}, respectively.
	%Assume further that
	%$V_{\del,h}\subset \cS_\del$ is a finite element space that contains all continuous piecewise  linear functions.\ Moreover, $V_{0, h}= \cS_0 \cap ({\bigcup}_{\del \in [0,\delta_0]} V_{\del,h})$.
	Then, for fixed $\si$,
	%and $\tau_h$,
	$\|u_{\si, h}-u_{\si} \|_{\cT_{\si}} $ satisfies \eqref{eq:gammaconvergence2} and converges to $0$ as $h \to 0$.
\end{theorem}

Note that the convergence result in the previous theorem also holds by \Cref{rmk:weakimpliesstr}.

		\section{Conclusion}

        We have presented a framework for asymptotically compatible schemes that applies to a wide variety of parametrized nonlinear variational problems, whose limits as the parameter $\delta \to 0$ are interpreted via Gamma-convergence. 
        %Not only have we demonstrated their efficacy for nonlinear nonlocal problems with horizon parameter $\delta$, the framework applies to other classes of parametrized problems.
        %We have demonstrated their efficacy for nonlinear nonlocal problems with horizon parameter $\delta$, but the framework applies to other classes of parametrized problems.
        %
        %In the linear case, the assumptions in \Cref{sec:LinearExample} generalize those in \cite{tian2014sinum,tian2020sirv} in similar contexts. We have not compared the assumptions in \Cref{sec:LinearExample} directly to the assumptions in \Cref{sec:ac} for the case of quadratic functionals, and we will leave this for a future work.
        %
        %We also leave to a future work an investigation of additional conditions on $\cF_\delta$ that implies the $\cT_\si$-strong convergence of discrete minimizers $u_{\si,h} \to u_{\si}$. 
        %
        We have also rigorously demonstrated the efficacy of the AC schemes in two examples of nonlinear nonlocal problems; the first with a horizon parameter that exhibits heterogeneous localization, and the second with a domain that varies with the nonlocal parameter.
        %However, the general framework of \Cref{sec:ac} readily applies to other examples, such as nonlocal problems posed on a domain that varies with nonlocal parameter. For instance, although the solution spaces $\cT_\si$ in \Cref{sec:example} satisfy a uniform embedding $\cT_\infty \subset \cT_\si$ for all $\si < \infty$, 
        %such a property may not hold in the case of minimization problems posed on a varying domain $\Omega_\delta$, i.e. the energy may be of the form
%        \begin{equation*}
%            \frac{1}{p} \int_{\Omega_\delta} \int_{\Omega_\delta} \rho_\delta(|\bx-\by|)  \left| \frac{u(\bx)-u(\by)}{\delta} \right|^p \, \rmd \bx \, \rmd \by - \int_{\Omega_\delta} f u \, \rmd \bx.
%        \end{equation*}
        %We leave the rigorous investigation of such problems for a future work.

        %Additionally, to place the nonlocal example in this work into context with previous works, we have analyzed a weak formulation of a linear nonlocal operator with heterogeneous localization. The assumptions in the linear framework described in \Cref{sec:Linear} generalize those in \cite{tian2014sinum,tian2020sirv}. 
        The examples treated in this work can be generalized along various lines while retaining all of the desired analytical properties. For instance, kernels $\rho$ can be chosen from more general classes; see for instance \cite{mengesha2015VariationalLimit,Du2023ME}. 
        
        Efforts to apply this more general framework to nonlocal models that incorporate boundary conditions via penalty terms, as well as nonlocal correspondence models \cite{Silling2007,han2023nonlocal} are ongoing. Studies of important issues for practical applications such as numerical quadrature formulae and fast linear and nonlinear solvers are surely worthy further investigation. Additional interesting questions to be further studied include conditional AC schemes \cite{tian2014sinum} and the general theory of AC schemes for gradient flows and other time-dependent problems as well as other types of nonlocal models and discretizations \cite{d2022prescription,trask2019asymptotically,leng2022petrov,glusa2023asymptotically,luo2023asymptotically,pang2022accurate,wang2022stability,yu2021asymptotically,you2020asymptotically,zhang2018accurate,ydu2023numerical}.
  
		%\section*{Acknowledgements}
		%(paper id/time). 

	%	\bibliographystyle{siamplain}
 \bibliographystyle{amsplain}
		\bibliography{References2}

\providecommand{\bysame}{\leavevmode\hbox to3em{\hrulefill}\thinspace}
\providecommand{\MR}{\relax\ifhmode\unskip\space\fi MR }
\providecommand{\MRhref}[2]{%
  \href{http://www.ams.org/mathscinet-getitem?mr=#1}{#2}
}
\providecommand{\href}[2]{#2}
\begin{thebibliography}{10}

\bibitem{arnold1997locking}
Douglas Arnold and Franco Brezzi, \emph{Locking-free finite element methods for
  shells}, Mathematics of Computation \textbf{66} (1997), no.~217, 1--14.

\bibitem{bobaru2016handbook}
Florin Bobaru, John~T Foster, Philippe~H Geubelle, and Stewart~A Silling,
  \emph{Handbook of peridynamic modeling}, CRC press, 2016.

\bibitem{BBM}
Jean Bourgain, Ha{\"i}m Brezis, and Petru Mironescu, \emph{Another look at
  {S}obolev spaces}, Optimal Control and Partial Differential Equations: In
  Honour of Professor Alain Bensoussan's 60th Birthday (Jos{\'e}~Luis Menaldi,
  Edmundo Rofman, and Agnes Sulem, eds.), IOS Press, 2001, pp.~439--455.

\bibitem{braides2002gamma}
Andrea Braides, \emph{Gamma-convergence for beginners}, vol.~22, Clarendon
  Press, 2002.

\bibitem{brenner2008finite}
Susanne~C Brenner and L~Ridgway Scott, \emph{Finite element multigrid methods},
  The Mathematical Theory of Finite Element Methods (2008), 155--173.

\bibitem{brezzi1980finite}
Franco Brezzi, Jacques Rappaz, and Pierre-Arnaud Raviart, \emph{Finite
  dimensional approximation of nonlinear problems: Part i: Branches of
  nonsingular solutions}, Numerische Mathematik \textbf{36} (1980), no.~1,
  1--25.

\bibitem{ciarlet2002finite}
Philippe~G Ciarlet, \emph{The finite element method for elliptic problems},
  SIAM, 2002.

\bibitem{diehl2019review}
Patrick Diehl, Serge Prudhomme, and Martin L{\'e}vesque, \emph{A review of
  benchmark experiments for the validation of peridynamics models}, Journal of
  Peridynamics and Nonlocal Modeling \textbf{1} (2019), 14--35.

\bibitem{Du2019book}
Qiang Du, \emph{Nonlocal modeling, analysis, and computation}, SIAM, 2019.

\bibitem{du2022fractional}
Qiang Du, Tadele Mengesha, and Xiaochuan Tian, \emph{Fractional {H}ardy-type
  and trace theorems for nonlocal function spaces with heterogeneous
  localization}, Analysis and Applications \textbf{20} (2022), no.~03,
  579--614.

\bibitem{Du2023ME}
Qiang Du, Tadele Mengesha, and Xiaochuan Tian, \emph{${L}^{p}$ compactness
  criteria with an application to variational convergence of some nonlocal
  energy functionals}, Mathematics in Engineering \textbf{5} (2023), no.~6,
  1--31.

\bibitem{ydu2023numerical}
Yu~Du and Jiwei Zhang, \emph{Numerical solutions for nonlocal wave equations by
  perfectly matched layers ii: The two-dimensional case}, Journal of
  Computational Physics \textbf{488} (2023), 112209.

\bibitem{d2022prescription}
Marta D’Elia and Yue Yu, \emph{On the prescription of boundary conditions for
  nonlocal poisson’s and peridynamics models}, Research in Mathematics of
  Materials Science, Springer, 2022, pp.~185--207.

\bibitem{ern2004theory}
Alexandre Ern and Jean-Luc Guermond, \emph{Theory and practice of finite
  elements}, vol. 159, Springer, 2004.

\bibitem{foss2022convergence}
Mikil Foss, Petronela Radu, and Yue Yu, \emph{Convergence analysis and
  numerical studies for linearly elastic peridynamics with dirichlet-type
  boundary conditions}, Journal of Peridynamics and Nonlocal Modeling (2022),
  1--36.

\bibitem{giusti2003direct}
Enrico Giusti, \emph{Direct methods in the calculus of variations}, World
  Scientific, 2003.

\bibitem{glusa2023asymptotically}
Christian Glusa, Marta D’Elia, Giacomo Capodaglio, Max Gunzburger, and
  Pavel~B Bochev, \emph{An asymptotically compatible coupling formulation for
  nonlocal interface problems with jumps}, SIAM Journal on Scientific Computing
  \textbf{45} (2023), no.~3, A1359--A1384.

\bibitem{guermond2010asymptotic}
Jean-Luc Guermond and Guido Kanschat, \emph{Asymptotic analysis of upwind
  discontinuous galerkin approximation of the radiative transport equation in
  the diffusive limit}, SIAM Journal on Numerical Analysis \textbf{48} (2010),
  no.~1, 53--78.

\bibitem{gunzburger1996finite}
Max~D Gunzburger and L~Steven Hou, \emph{Finite-dimensional approximation of a
  class of constrained nonlinear optimal control problems}, SIAM journal on
  control and optimization \textbf{34} (1996), no.~3, 1001--1043.

\bibitem{han2023nonlocal}
Zhaolong Han and Xiaochuan Tian, \emph{Nonlocal half-ball vector operators on
  bounded domains: {{Poincar{\'e}}} inequality and its applications},
  Mathematical Models and Methods in Applied Sciences \textbf{33} (2023),
  no.~12, 2507--2556.

\bibitem{jha2021finite}
PK~Jha and R~Lipton, \emph{Finite element approximation of nonlocal dynamic
  fracture models}, Discrete and Continuous Dynamical Systems-Series B
  \textbf{26} (2021), no.~3, 1675.

\bibitem{jha2019finite}
Prashant~K Jha and Robert Lipton, \emph{Finite differences and finite elements
  in nonlocal fracture modeling: A priori convergence rates}, Handbook of
  Nonlocal Continuum Mechanics for Materials and Structures. Springer (2019),
  1457--1494.

\bibitem{leng2022petrov}
Yu~Leng, Xiaochuan Tian, Leszek Demkowicz, Hector Gomez, and John~T Foster,
  \emph{A {Petrov-Galerkin} method for nonlocal convection-dominated diffusion
  problems}, Journal of Computational Physics \textbf{452} (2022), 110919.

\bibitem{leng2020asymptotically}
Yu~Leng, Xiaochuan Tian, Nathaniel~A Trask, and John~T Foster,
  \emph{Asymptotically compatible reproducing kernel collocation and meshfree
  integration for the peridynamic navier equation}, Computer Methods in Applied
  Mechanics and Engineering \textbf{370} (2020), 113264.

\bibitem{luo2023asymptotically}
Wangbo Luo and Yanxiang Zhao, \emph{Asymptotically compatible schemes for
  nonlocal {Ohta Kawasaki} model}, arXiv preprint arXiv:2311.15186 (2023).

\bibitem{MengeshaDuElasticity}
Tadele Mengesha and Qiang Du, \emph{Nonlocal constrained value problems for a
  linear peridynamic {N}avier equation}, J. Elasticity \textbf{116} (2014),
  no.~1, 27--51. \MR{3206170}

\bibitem{mengesha2015VariationalLimit}
Tadele Mengesha and Qiang Du, \emph{On the variational limit of a class of
  nonlocal functionals related to peridynamics}, Nonlinearity \textbf{28}
  (2015), no.~11, 3999.

\bibitem{Mengesha2023Linearization}
Tadele Mengesha and James~M Scott, \emph{Linearization and localization of
  nonconvex functionals motivated by nonlinear peridynamic models}, arXiv
  preprint arXiv:2306.15446 (2023).

\bibitem{pang2022accurate}
Gang Pang, Songsong Ji, and Xavier Antoine, \emph{Accurate absorbing boundary
  conditions for two-dimensional peridynamics}, Journal of Computational
  Physics \textbf{466} (2022), 111351.

\bibitem{ponce2004new}
Augusto~C Ponce, \emph{A new approach to sobolev spaces and connections to
  $\gamma$-convergence}, Calc. Var. Partial Differential Equations \textbf{19}
  (2004), no.~3, 229--255.

\bibitem{quarteroni2007numerical}
Alfio Quarteroni and Gianluigi Rozza, \emph{Numerical solution of parametrized
  navier--stokes equations by reduced basis methods}, Numerical methods for
  partial differential equations: an international journal \textbf{23} (2007),
  no.~4, 923--948.

\bibitem{scott2023nonlocal}
James~M Scott and Qiang Du, \emph{Nonlocal problems with local boundary
  conditions {I}: function spaces and variational principles}, arXiv preprint
  arXiv:2307.08855 (2023).

\bibitem{scott2023nonlocal2}
\bysame, \emph{Nonlocal problems with local boundary conditions {II}: {G}reen's
  identities and regularity of solutions}, arXiv preprint arXiv:2308.05180
  (2023).

\bibitem{Silling2000}
S.~A. Silling, \emph{Reformulation of elasticity theory for discontinuities and
  long-range forces}, J. Mech. Phys. Solids \textbf{48} (2000), no.~1,
  175--209. \MR{1727557}

\bibitem{Silling2007}
S.~A. Silling, M.~Epton, O.~Weckner, J.~Xu, and E.~Askari, \emph{Peridynamic
  states and constitutive modeling}, J. Elasticity \textbf{88} (2007), no.~2,
  151--184. \MR{2348150}

\bibitem{silling2017stability}
SA~Silling, \emph{Stability of peridynamic correspondence material models and
  their particle discretizations}, Computer Methods in Applied Mechanics and
  Engineering \textbf{322} (2017), 42--57.

\bibitem{tao2019nonlocal}
Yunzhe Tao, Xiaochuan Tian, and Qiang Du, \emph{Nonlocal models with
  heterogeneous localization and their application to seamless local-nonlocal
  coupling}, Multiscale Modeling \& Simulation \textbf{17} (2019), no.~3,
  1052--1075.

\bibitem{tian2013sinum}
Xiaochuan Tian and Qiang Du, \emph{Analysis and comparison of different
  approximations to nonlocal diffusion and linear peridynamic equations}, SIAM
  Journal on Numerical Analysis \textbf{51} (2013), no.~6, 3458--3482.

\bibitem{tian2014sinum}
\bysame, \emph{Asymptotically compatible schemes and applications to robust
  discretization of nonlocal models}, SIAM Journal on Numerical Analysis
  \textbf{52} (2014), no.~4, 1641--1665.

\bibitem{tian2014asymptotically}
\bysame, \emph{Asymptotically compatible schemes and applications to robust
  discretization of nonlocal models}, SIAM Journal on Numerical Analysis
  \textbf{52} (2014), no.~4, 1641--1665.

\bibitem{tian2017trace}
\bysame, \emph{Trace theorems for some nonlocal function spaces with
  heterogeneous localization}, SIAM Journal on Mathematical Analysis
  \textbf{49} (2017), no.~2, 1621--1644.

\bibitem{tian2020sirv}
\bysame, \emph{Asymptotically compatible schemes for robust discretization of
  parametrized problems with applications to nonlocal models}, SIAM Review
  \textbf{62} (2020), no.~1, 199--227.

\bibitem{trask2019asymptotically}
Nathaniel Trask, Huaiqian You, Yue Yu, and Michael~L Parks, \emph{An
  asymptotically compatible meshfree quadrature rule for nonlocal problems with
  applications to peridynamics}, Computer Methods in Applied Mechanics and
  Engineering \textbf{343} (2019), 151--165.

\bibitem{wang2022stability}
Jihong Wang, Jiwei Zhang, and Chunxiong Zheng, \emph{Stability and error
  analysis for a second-order approximation of 1d nonlocal schr{\"o}dinger
  equation under {DtN}-type boundary conditions}, Mathematics of Computation
  \textbf{91} (2022), no.~334, 761--783.

\bibitem{you2020asymptotically}
Huaiqian You, XinYang Lu, Nathaniel Task, and Yue Yu, \emph{An asymptotically
  compatible approach for neumann-type boundary condition on nonlocal
  problems}, ESAIM: Mathematical Modelling and Numerical Analysis \textbf{54}
  (2020), no.~4, 1373--1413.

\bibitem{yu2021asymptotically}
Yue Yu, Huaiqian You, and Nathaniel Trask, \emph{An asymptotically compatible
  treatment of traction loading in linearly elastic peridynamic fracture},
  Computer Methods in Applied Mechanics and Engineering \textbf{377} (2021),
  113691.

\bibitem{zhang2018accurate}
Xiaoping Zhang, Jiming Wu, and Lili Ju, \emph{An accurate and asymptotically
  compatible collocation scheme for nonlocal diffusion problems}, Applied
  Numerical Mathematics \textbf{133} (2018), 52--68.

\end{thebibliography}

\end{document}